\documentclass[11pt,a4paper,reqno]{amsart}
\usepackage[english]{babel}
\usepackage[applemac]{inputenc}
\usepackage[T1]{fontenc}
\usepackage{palatino}
\usepackage{cite}
\usepackage{esint}
\usepackage{verbatim}
\usepackage{amsmath}
\usepackage{amssymb}
\usepackage{amsthm}
\usepackage{amsfonts}
\usepackage[pdftex]{graphicx}
\usepackage{mathtools}
\usepackage{enumitem}
\usepackage{mathabx}

\usepackage[colorlinks = true, citecolor = black]{hyperref}
\title{On the Assouad dimension of projections}
\author{Tuomas Orponen}
\address{University of Helsinki, Department of Mathematics and Statistics}
\thanks{T.O. is supported by the Academy of Finland via the project \emph{Quantitative rectifiability in Euclidean and non-Euclidean spaces}, grant No. 309365.}
\email{tuomas.orponen@helsinki.fi}

\newcommand{\R}{\mathbb{R}}

\newcommand{\N}{\mathbb{N}}

\newcommand{\Z}{\mathbb{Z}}

\newcommand{\calT}{\mathcal{T}}

\newcommand{\calD}{\mathcal{D}}
\newcommand{\calH}{\mathcal{H}}

\newcommand{\calB}{\mathcal{B}}
\newcommand{\calG}{\mathcal{G}}

\newcommand{\calS}{\mathcal{S}}

\newcommand{\calR}{\mathcal{R}}

\newcommand{\spt}{\operatorname{spt}}
\newcommand{\Hd}{\dim_{\mathrm{H}}}
\newcommand{\Pd}{\dim_{\mathrm{p}}}

\newcommand{\diam}{\operatorname{diam}}
\newcommand{\card}{\operatorname{card}}

\newcommand{\Ad}{\dim_{\textsc{A}}}

\newcommand{\calM}{\mathcal{M}}

\numberwithin{equation}{section}

\theoremstyle{plain}
\newtheorem{thm}[equation]{Theorem}

\newtheorem{conjecture}[equation]{Conjecture}
\newtheorem{lemma}[equation]{Lemma}

\newtheorem{proposition}[equation]{Proposition}

\theoremstyle{definition}

\newtheorem{definition}[equation]{Definition}

\theoremstyle{remark}
\newtheorem{remark}[equation]{Remark}

\newcommand{\nref}[1]{(\hyperref[#1]{#1})}

\begin{document}

\begin{abstract} Let $F \subset \R^{2}$, and let $\Ad$ stand for Assouad dimension. I prove that $$\Ad \pi_{e}(F) \geq \min\{\Ad F,1\}$$ for all $e \in S^{1}$ outside of a set of Hausdorff dimension zero. This is a strong variant of Marstrand's projection theorem for Assouad dimension, whose analogue is not true for other common notions of fractal dimension, such as Hausdorff or packing dimension. \end{abstract}

\maketitle

\tableofcontents

\section{Introduction}

\subsection{The main result and previous work} For $F \subset \R^{2}$, let $\Ad F$ be the \emph{Assouad dimension} of $F$, see Definition \ref{Adim}. For $e \in S^{1}$, write $\pi_{e} \colon \R^{2} \to \R$ for the projection map $\pi_{e}(x) = x \cdot e$. Here is the main result of the paper:

\begin{thm}\label{main} Let $F \subset \R^{2}$. Then, 
\begin{displaymath} \Hd \{e \in S^{1} : \Ad \pi_{e}(F) < \min\{\Ad F,1\}\} = 0. \end{displaymath}
\end{thm}
Theorem \ref{main} improves on an earlier result of Fraser and the author \cite[Theorem 2.1]{FO}, where it was proven that
\begin{equation}\label{FOResult} \mathcal{H}^{1}(\{e \in S^{1} : \Ad \pi_{e}(F) < \min\{\Ad F,1\}) = 0. \end{equation}
To put Theorem \ref{main} into proper context, I briefly list below the main existing projection theorems concerning a general compact set $K \subset \R^{2}$, and I also discuss their sharpness. 
\subsubsection*{Hausdorff dimension} Let $\Hd$ stand for Hausdorff dimension. Marstrand in 1954, see \cite[Theorem II]{Mar}, proved that the set 
\begin{displaymath} E := E(K) := \{e \in S^{1} : \Hd \pi_{e}(K) < \min\{\Hd K,1\}\} \end{displaymath}
has $\calH^{1}(E) = 0$. In 1968, Kaufman \cite{Ka} improved this in the case $\Hd K < 1$ by showing that $\Hd E \leq \Hd K$. Kaufman's bound is sharp in the following sense: Kaufman and Mattila \cite[Theorem 5]{KM} constructed a compact set $K \subset \R^{2}$ with any Hausdorff dimension $\Hd K \in [0,1]$ such that $\Hd E = \Hd K$.
\subsubsection*{Packing and box dimensions} Let $\Pd$ stand for packing dimension. Then, there are compact sets $K \subset \R^{2}$ such that
\begin{displaymath} \{e \in S^{1} : \Pd \pi_{e}(K) < \min\{\Pd K,1\}\} = S^{1}. \end{displaymath}
Such sets were first constructed by J\"arvenp\"a\"a \cite{Ja}. However, positive results can be obtained by considering instead $\{e \in S^{1} : \Pd \pi_{e}(K) < s\}$ for various $0 < s < \Pd K$. For sharp results, both positive and negative, see the work \cite{FH} of Falconer and Howroyd. The situation is similar for box dimension(s), see the references above.
\subsubsection*{Assouad dimension} Theorem \ref{main} evidently gives a sharp result for Assouad dimension, although one could further ask if 
\begin{displaymath} \Pd \{e \in S^{1} : \Ad \pi_{e}(K) < \min\{\Ad K,1\}\} = 0. \end{displaymath}
The proof in this paper does not seem to give this improvement. It would also be interesting to know if $\{e \in S^{1} : \Ad \pi_{e}(K) < \min\{\Ad K,1\}\}$ can be uncountable. 

Theorem \ref{main} does not imply that $\Ad \pi_{e}(K) = \min\{\Ad K,1\}$ for all $e \in S^{1}$ outside of a small set of exceptions. In fact, such a statement is far from true. It was already observed in \cite{FO} that the map $e \mapsto \Ad \pi_{e}(K)$ can be essentially non-constant. This observation was recently strengthened by Fraser and K\"aenm\"aki \cite{FK}: if $\phi \colon S^{1} \to [0,1]$ is any upper semicontinuous function with $\phi(e) = \phi(-e)$, then there exists a compact set $K \subset \R^{2}$ with $\Ad K = 0$ such that $\Ad \pi_{e}(K) = \phi(e)$ for all $e \in S^{1}$.
\subsubsection*{A mixed problem} What about the set
\begin{displaymath} E' := \{e \in S^{1} : \Pd \pi_{e}(K) < \min\{\Hd K,1\}\}? \end{displaymath} 
Since packing dimension an upper bound for Hausdorff dimension, the theorems of Kaufman and Marstrand imply that $\calH^{1}(E') = 0$ and $\Hd E' \leq \Hd K$. This is unlikely to be sharp: I am not aware of a compact set $K \subset \R^{2}$ with $\Hd E' > 0$! In contrast, $\Pd E'$ can take values arbitrarily close to $1$, see \cite[Theorem 1.17]{Or1}.

\subsubsection*{Sets with additional structure} If $K \subset \R^{2}$ is self-similar, then
\begin{equation}\label{hochman} \Pd \{e \in S^{1} : \Hd \pi_{e}(K) < \min\{\Hd K,1\}\} = 0. \end{equation}
This is a result of Hochman \cite[Theorem 1.8]{Ho} in the case where $K$ contains no irrational rotations. In the presence of irrational rotations, one can further improve \eqref{hochman} to
\begin{equation}\label{PeresShmerkin} \{e \in S^{1} : \Hd \pi_{e}(K) < \min\{\Hd K,1\}\} = \emptyset, \end{equation}
which is an earlier result of Peres and Shmerkin \cite{PS}. Under suitable irrationality hypotheses, too lengthy to explain here, the conclusion \eqref{PeresShmerkin} is also known for self-conformal sets \cite{BJ}, and several classes of self-affine sets, see \cite{FFS,BHR}.

\subsection{Outline of the proof}\label{outline} Before explaining the main steps in the proof of Theorem \ref{main}, I need to describe two initial reductions. 

\subsubsection{Initial reductions}\label{reductions} First, since Assouad dimension is invariant under taking closures, it suffices to prove Theorem \ref{main} for closed sets $F$. Second, it suffices to prove Theorem \ref{main} for compact sets $F \subset \R^{2}$ with $\calH^{\Ad F}(F) > 0$. The proof of this reduction is the same as the proof of \cite[Theorem 2.9]{Fr}, but I sketch the idea briefly; see \cite{Fr} for more details. By a result of K\"aenm\"aki, Ojala, and Rossi \cite[Proposition 5.7]{KOR}, any closed set $F \subset \R^{2}$ has a \emph{weak tangent} $E \subset \R^{2}$ with $\Ad E = \Hd E = \Ad F$, and even
\begin{equation}\label{KORResult} \calH^{\Ad E}(E) > 0. \end{equation}
The authors of \cite{KOR} omit mentioning \eqref{KORResult}, but this is what they prove (see also the discussion after \cite[Theorem 1.3]{Fr}). Now, for any $e \in S^{1}$, the projection $\pi_{e}(E)$ turns out to be a \textbf{subset} of some weak tangent $W_{e}$ of $\pi_{e}(F)$ (see the proof of \cite[Theorem 2.9]{Fr} for details), and hence 
\begin{displaymath} \Ad \pi_{e}(F) \geq \Ad W_{e} \geq \Ad \pi_{e}(E) \end{displaymath}
where the first inequality is \cite[Proposition 6.1.5]{MT}. Consequently,
\begin{displaymath} \{e \in S^{1} : \Ad \pi_{e}(F) < \min\{\Ad F,1\}\} \subset \{e \in S^{1} : \Ad \pi_{e}(E) < \min\{\Ad E,1\}\}, \end{displaymath}
and Theorem \ref{main} now follows if one manages to prove that the set on the right has zero Hausdorff dimension. Recalling that $E$ satisfies \eqref{KORResult}, this completes the proof of the second reduction.

\subsubsection{The main argument}\label{mainOutline} Let $F \subset \R^{2}$ be a compact set satisfying $\calH^{d}(F) > 0$, where $d := \Ad F$. These are reasonable assumptions by the previous discussion. Then, let $\mu$ be a $d$-dimensional Frostman measure supported on $F$, and assume with no loss of generality that $F = \spt \mu$. The measure $\mu$ is not quite $d$-regular, but not too far from it either, precisely because $\Ad F$ matches the Frostman exponent of $\mu$. For a way to quantify this, see Lemma \ref{lemma1}.

The measure $\mu$ itself is still too general to work with, so we need to pass to another tangent $\nu = \mu^{B}$, where $B$ is a ball with $\mu(B) \approx \diam(B)^{d}$. Most balls have this property by the near-$d$-regularity of $\mu$. To list the (less trivial) properties required of $\nu$, start with a counter assumption: $\Hd S > \epsilon > 0$, where $S = \{e : \Ad \pi_{e}(F) < D\}$ and $0 < D < \min\{1,d\}$. Then, locate an $\epsilon$-dimensional Frostman measure $\sigma$ on $S$. The properties needed of $\nu$ are now -- very roughly speaking! -- the following: there is a constant $0 \leq s \leq D$ such that
\begin{itemize}
\item[(a)] $\pi_{e}(\nu)$ is exact dimensional with dimension $s$ for $\sigma$ almost every $e \in S$,
\item[(b)] The projections $\pi_{e}$ are dimension conserving relative to $\nu$ (in the sense of Furstenberg \cite{Fu}) for $\sigma$ almost every $e \in S$.
\end{itemize}
The second requirement means that the measure $\nu$ conditioned on a $\pi_{e}(\nu)$-generic fibre $\pi_{e}^{-1}\{x\}$ is at least $(d - s)$-dimensional. It is possible that a tangent $\nu = \mu^{B}$ satisfying (a)-(b) \textbf{literally} could be extracted by the theory of CP-chains, see \cite[Section 6]{Fu}, \cite[Theorem 1.22]{Ho1}, and \cite[Theorem 1.30]{Ho1}. However, the requirements (a)-(b) should not be interpreted literally: what we really need are certain $\delta$-discretised versions of (a)-(b); for a precise statement (which is admittedly difficult to decipher with the current background), see \nref{K1}-\nref{K2} in Section \ref{refining}. So, instead of applying the theory of CP-chains, the proof below only relies on combinatorial argument, notably the pigeonhole principle. 

After $\nu$ has been found, we start looking for a contradiction to the hypothesis that $\Hd \sigma > 0$. This constitutes the main effort in the paper. Note that $\nu = \mu^{B}$ is still near-$d$-regular, because $\mu(B) \approx \diam(B)^{d}$, and $\Ad F = d$. So, what we roughly need to prove is the following:
\begin{conjecture}\label{conjecture1} Assume that $\nu$ is a near-$d$-regular measure on $\R^{2}$, $0 \leq s < d$, and $\sigma$ is a Borel probability measure on $S^{1}$ such that (a)-(b) are satisfied. Then $\Hd \sigma = 0$.
\end{conjecture}
Conjecture \ref{conjecture1} seems plausible, but I do not claim to prove it here. In fact, recalling that the our $\nu$ only satisfies approximate variants of (a)-(b), Conjecture \ref{conjecture1} would not be literally useful in the present context. However, the underlying point in Section \ref{core} is to prove a version of Conjecture \ref{conjecture1}, using the "real" information we have about $\nu$, and hence contradict the positive-dimensionality of $\sigma$. 

To be honest, this "real" information contains some pieces not contained in (a)-(b). First, we have $S \subset \{e : \Ad \pi_{e}(F) < D\}$, which in particular implies a quantitative -- and useful -- porosity property for $\spt \pi_{e}(\nu)$, $e \in S$. With additional effort, one might be able to work with the weaker measure-theoretic porosity of $\pi_{e}(\nu)$ implied by (a) alone, but the set-theoretic porosity of $\spt \pi_{e}(\nu)$ is certainly more pleasant to apply. A second, and more crucial, piece of additional information is
\begin{itemize}
\item[(a')] property (a) also for all tangents of $\nu$ of "at moderate scales". 
\end{itemize}
This roughly means that if $\delta > 0$ is the smallest scale where all the action happens, and $\delta < \Delta \leq \delta^{\kappa}$ for some suitable (small) constant $\kappa = \kappa(\epsilon) > 0$, then the renormalised restriction of $\nu$ to any $\Delta$-ball centred at $\spt \nu$ has roughly $s$-dimensional projections at scale $\delta/\Delta$ for most directions $e \in S$. This is vital in Section \ref{core}, but makes virtually no difference in the construction of $\nu$. I do not know how to derive -- or even formulate -- an analogous statement from/within the theory of CP-chains.

At the end, the proof of our (discretised and watered-down version of) Conjecture \ref{conjecture1} rests on an application of Shmerkin's inverse theorem \cite[Theorem 2.1]{Sh}. This theorem is the latest quantification of the following phenomenon, initially discovered by Bourgain \cite{Bo1,Bo2}, and later developed by Hochman \cite{Ho}: if $\nu_{1} \times \nu_{2}$ is a product measure on $\R^{2}$, $e \in S^{1}$ is at positive distance from $\{(1,0),(0,1)\}$, and the $\delta$-entropies of $\pi_{e}(\nu_{1} \times \nu_{2})$ and $\dim \nu_{1}$ are comparable for some $0 < \delta \ll 1$, then all the scales between $\delta$ and $1$ can be split into two disjoint groups: those where $\nu_{1}$ is "uniform", and those where $\nu_{2}$ is "singular".

In our setting, there are no product measures to begin with. However, assuming that $(1,0) \in S$ without loss of generality, a scheme introduced in \cite{Or} allows one to derive from $\nu$ -- using (a') -- a product measure $\nu_{1} \times \nu_{2}$ with the properties that
\begin{enumerate}
\item $\nu_{1} \approx \pi_{(1,0)}(\nu)$,
\item $\nu_{2} \approx \nu$ conditioned on a $\pi_{(1,0)}(\nu)$-generic fibre $\pi_{(1,0)}^{-1}\{x\}$,
\item $\dim \pi_{e}(\nu_{1} \times \nu_{2}) \approx \dim \pi_{e}(\nu)$ for $e \in S$ sufficiently close to $(1,0)$.
\end{enumerate}
This step is accomplished in Sections \ref{s:productStructure}-\ref{productMeasures}. The main geometric idea is that if $\delta > 0$ is a scale, $T = J \times [0,1] \subset [0,1]^{2}$ is a vertical $\delta^{1/2}$-tube, where $J \subset [0,1]$ is a $\delta^{1/2}$-interval, and $P \subset T$ is any $\delta$-separated set, then there always exists a "quasi-product" set $\bar{P} \subset T$ of the form
\begin{displaymath} \bar{P} = \bigcup_{h \in \calD} \bar{P}_{h} \times \{h\} \end{displaymath}
such that the $\delta$-covering numbers of $\pi_{e}(P)$ and $\pi_{e}(\bar{P})$ are comparable for all $e \in S^{1}$ with $|e - e_{0}| \leq \delta^{1/2}$, where $e_{0} := (1,0)$. Here $\calD \subset [0,1]$ is a $\delta^{1/2}$-net in $\pi_{(0,1)}(P)$, and each $\bar{P}_{h}$ is a $\delta$-separated subset of $J$. This idea already appeared in \cite{Or}, and \cite[Section 1.3]{Or} contains a little more explanation. To remove the word "quasi", we would need to know that the sets $\bar{P}_{h}$ are the same for (nearly) all $h \in \mathcal{D}$. This is generally not true, but a reasonable substitute holds in a situation where the $\delta$-covering number of $\pi_{e_{0}}(\bar{P})$ is comparable to $|\bar{P}_{h}|$ for nearly all $h \in \mathcal{D}$. Since $\bar{P}_{h} \subset \pi_{e_{0}}(\bar{P})$ for all $h \in \mathcal{D}$, this situation implies that $|\bar{P}_{h} \cap \bar{P}_{h'}| \approx |\bar{P}_{h}|$ for nearly all pairs $h,h' \in \mathcal{D}$. This information is almost as good as knowing that $\bar{P}_{h} = \bar{P}_{h'}$ for all $h,h' \in \mathcal{D}$. In conclusion, seriously cutting corners, we might say that $\bar{P}$ is a product set whose projections in directions $|e - e_{0}| \leq \delta^{1/2}$ have $\delta$-covering numbers comparable to those of $\pi_{e}(P)$. I hope this sounds remotely like (3) above. Finally, the crucial comparability of $|\bar{P}_{h}|$ and $\pi_{e_{0}}(\bar{P})$ in our concrete situation is based on (a').

The claims (1)-(3) should not be taken literally; the first one in particular is quite far from reality. Let us, nevertheless, argue that having them would be useful in completing the proof of Theorem \ref{main}. For $e \in S \, \setminus \, \{(1,0),(0,1)\}$, we have
\begin{displaymath} \dim \pi_{e}(\nu_{1} \times \nu_{2}) \approx \dim \pi_{e}(\nu) = s = \dim \pi_{(1,0)}(\nu) \approx \dim \nu_{1}, \qquad e \in S. \end{displaymath}
This would be useless if $S = \{(1,0)\}$, but the counter assumption $\dim S > 0$ allows us to pick $e$, as above, at a reasonable distance from $\{(1,0),(0,1)\}$. Hence, Shmerkin's inverse theorem describes the structure of $\nu_{1}$ and $\nu_{2}$. Since $\nu_{1} \approx \pi_{(1,0)}(\nu)$ has the quantitative porosity property alluded to above, $\nu_{1}$ cannot be "uniform" on \textbf{any} scales, and hence $\nu_{2}$ is "singular" on all scales. This forces $\dim \nu_{2} \approx 0$. But it follows from the second bullet point above, and (b), that actually $\dim \nu_{2} \geq d - s > 0$. This gives the desired contradiction.

The detailed proof given below is completely elementary and self-contained, except for the application of Shmerkin's inverse theorem at the end.

\section{Acknowledgements}

I thank Tom Kempton for useful discussions during an early phase of the project. I also thank the anonymous reviewer for reading the paper carefully, and for making a large number of helpful suggestions. 

\section{Finding a good blow-up}  

I will now start to implement the strategy outlined in Section \ref{outline}. Non-zero Radon measures supported on a set $E \subset \R^{2}$ will be denoted by $\calM(E)$. If $E \subset \R^{2}$ is bounded, and $r > 0$, the notation $N(E,r)$ stands for the smallest number of open balls of radius $r$ needed to cover $E$. All balls in the paper will be open, unless otherwise specified. The notation $|P|$ will refer to the cardinality of a finite set $P$. For $A,B > 0$, and a parameter "$p$", the notation $A \lesssim_{p} B$ means that there exists a constant $C \geq 1$, depending only on $p$, such that $A \leq CB$. The notation $A \lesssim B$ means that the constant $C$ is absolute. The notation $A \gtrsim B$ is equivalent to $B \lesssim A$, and $A \sim B$ is shorthand notation for $A \lesssim B \lesssim A$. 

\begin{definition}[Assouad dimension]\label{Adim} Let $n \geq 1$ and $F \subset \R^{n}$. The \emph{Assouad dimension} of $F$ is the infimum of the numbers $s \geq 0$ to which there corresponds a constant $C = C_{s} > 0$ as follows: 
\begin{displaymath} N(F \cap B(x,R),r) \leq C\left(\frac{R}{r} \right)^{s}, \qquad x \in \R^{n}, \: 0 < r \leq R < \infty. \end{displaymath}

\end{definition}
\begin{definition}[$d$-quasiregular measures] Let $d \in [0,2]$. A measure $\mu \in \calM(\R^{2})$ is called \emph{$d$-quasiregular} if $\mu$ is a $d$-Frostman measure with $\Ad (\spt \mu) \leq d$. In other words, for every $\epsilon > 0$ there is a constant $C_{\epsilon} \geq 1$ such that
\begin{equation}\label{form39} \mu(B(x,r)) \leq r^{d} \quad \text{and} \quad N([\spt \mu] \cap B(x,R),r) \leq C_{\epsilon}\left(\frac{R}{r} \right)^{d + \epsilon} \end{equation}
for all $x \in \R^{2}$ and $0 < r < R < \infty$. \end{definition}

In the sequel, I write $B_{0} := B(0,1)$. Recalling the argument in Section \ref{reductions}, Theorem \ref{main} is a consequence of following statement:

\begin{thm}\label{mainTechnical} Let $0 \leq d \leq 2$, and let $\mu \in \calM(B_{0})$ be a $d$-quasiregular measure. Then
\begin{displaymath} \Hd \{e \in S^{1} : \Ad \pi_{e}(\spt \mu) < \min\{d,1\}\} = 0. \end{displaymath}
\end{thm}
What follows is a proof of Theorem \ref{mainTechnical}. For the rest of the paper, fix $0 < d \leq 2$, and a $d$-quasiregular measure $\mu \in \calM(B_{0})$. Write $K := \spt \mu$. 

\subsection{Blow-ups and their (quasi)regularity} I now define what is meant by \emph{blowing up} of a measure in $\calM(\R^{2})$.

\begin{definition}[Blow-ups] Let $\nu \in \calM(\R^{2})$, and let $B = B(x,r) \subset \R^{2}$ be a ball. Let $T_{B}(y) = (y - x)/r$ be the unique homothetic map taking $B$ to $B_{0}$, and define $\nu^{B} \in \calM(\R^{2})$ as
\begin{displaymath} \nu^{B} := \nu^{B,d} := \frac{T_{B}(\nu)}{r^{d}}. \end{displaymath}
\end{definition}
In general, the definition above does not guarantee that $\nu^{B}(B_{0}) = 1$. However, we will only use the blow-up procedure in balls $B(x,r)$ on which the measure "$\nu$" in question has mass roughly $r^{d}$. I record the following "chain rule"
\begin{equation}\label{chain} (\mu^{B})^{B'} = \mu^{(T_{B'} \circ T_{B})^{-1}(B_{0})}, \end{equation} 
which is valid for all balls $B,B'\subset \R^{2}$ with radii $r,r' > 0$, and follows by noting that
\begin{displaymath} (\mu^{B})^{B'}(A) = \frac{\mu_{B}(T_{B'}^{-1}(A))}{(r')^{d}} = \frac{\mu((T_{B'} \circ T_{B})^{-1}(A))}{(rr')^{d}} = \frac{\mu(T_{(T_{B'} \circ T_{B})^{-1}(B_{0})}^{-1}(A))}{(rr')^{d}}, \qquad A \subset \R^{2}. \end{displaymath}

The next lemma verifies that $d$-quasiregularity is preserved under blow-ups.

\begin{lemma}\label{lemma2} If $B = B(x_{0},r_{0}) \subset \R^{2}$, then $\mu^{B}$ satisfies \eqref{form39} with the same constants as $\mu$:
\begin{displaymath} \mu^{B}(B(x,r)) \leq r^{d} \quad \text{and} \quad N(\spt \mu^{B} \cap B(x,R),r) \leq C_{\epsilon} \left(\frac{R}{r} \right)^{d + \epsilon} \end{displaymath}
for all $x \in \R^{2}$, $\epsilon > 0$, and $0 < r < R < \infty$. \end{lemma}
\begin{proof} For $x \in \R^{2}$, $\epsilon > 0$ and $0 < r < R < \infty$, note that
\begin{displaymath} \mu^{B}(B(x,r)) = \frac{\mu(B(x_{0} + r_{0}x,rr_{0}))}{r_{0}^{d}} \leq r^{d}, \end{displaymath}
and also
\begin{align*} N(\spt \mu^{B} \cap B(x,R),r) & = N(K \cap B(x_{0} + r_{0}x,r_{0}R),rr_{0}) \leq C_{\epsilon}\left(\frac{r_{0}R}{rr_{0}} \right)^{d - \epsilon} = C_{\epsilon} \left(\frac{R}{r} \right)^{d - \epsilon}, \end{align*}
recalling the notation $K = \spt \mu$. \end{proof}

Our life would be an $\epsilon$ easier if $\mu$ were $d$-regular, and not just $d$-quasiregular. However, the conditions \eqref{form39} together guarantee that "$\mu(B(x,r)) \approx r^{d}$ for most balls $B(x,r)$":

\begin{lemma}\label{lemma1} Let $\mu \in \calM(\R^{2})$ be a measure satisfying \eqref{form39}, and let $E \subset B_{0}$. Then, for $\epsilon > 0$ and $0 < r \leq R \leq 1$, the set
\begin{displaymath} E_{r,\epsilon}(R) := \{x \in E : \mu(E \cap B(x,R)) \leq r^{\epsilon}R^{d} \} \end{displaymath}
has measure $\mu(E_{r,\epsilon}(R)) \lesssim C_{\epsilon/2}r^{\epsilon/2}$.
\end{lemma}

\begin{proof}[Proof of Lemma \ref{lemma1}] Using the $5R$-covering theorem, see \cite[Theorem 1.2]{MR1800917}, choose a finite collection $\calB_{0}$ of balls of the form $B(x_{i},R)$ with $x_{i} \in E_{r,\epsilon}(R) \subset B_{0}$ which cover $E_{r,\epsilon}(R)$ and such that the balls $B(x_{i},R/5)$ are disjoint. Let 
\begin{displaymath} \calB := \{B(x_{i},R) \in \calB_{0} : B(x_{i},R) \cap \spt \mu \neq \emptyset\}, \end{displaymath}
and note that $\mu$ almost all of $E_{r,\epsilon}(R)$ is contained in $\cup \calB$. It now follows from \eqref{form39} that $|\calB| \lesssim C_{\epsilon/2}R^{-d - \epsilon/2}$. Indeed, start with $R$-ball cover $\calB'$ of $\spt \mu \cap B_{0}$ with $|\calB'| \leq C_{\epsilon/2}R^{-d - \epsilon/2}$. Then note that the balls $5B$ with $B \in \calB'$ cover the balls $B(x_{i},R/5)$ with $B(x_{i},R) \in \calB$, and conclude that $|\calB'| \gtrsim |\calB|$ by the disjointness of the balls $B(x_{i},R/5)$. Since the intersections $E \cap B(x_{i},R)$ with $B(x_{i},R) \in \calB$, cover $\mu$ almost all of $E_{r,\epsilon}(R) \subset E$, we infer that
\begin{displaymath} \mu(E_{r,\epsilon}(R)) \leq \sum_{B \in \calB} \mu(E \cap B(x_{i},R)) \leq |\calB| \cdot r^{\epsilon}R^{d} \lesssim C_{\epsilon/2}r^{\epsilon/2}, \end{displaymath} 
using $r \leq R$ in the last estimate. \end{proof}

\subsection{The measure $\sigma$ and the key constants of the paper} Let $\sigma$ be an arbitrary Borel probability measure on $S^{1}$. The reader is advised to think that $\sigma$ is a Frostman measure in the set
\begin{equation}\label{form79} \{e \in S^{1} : \Ad \pi_{e}(K) < \min\{d,1\}\}, \end{equation}
but no Frostman condition will be required for a long time to come. Write
\begin{displaymath} 2^{-\N} := \{\tfrac{1}{2},\tfrac{1}{4},\ldots\}. \end{displaymath}
For $r > 0$ and $e \in S^{1}$, an \emph{$(r,e)$-tube} stands for a set of the form $B_{0} \cap \pi_{e}^{-1}(I)$, where $I \subset \R$ is an interval of length $r$. The tube $B_{0} \cap \pi_{e}^{-1}(I)$ is \emph{dyadic} if 
\begin{displaymath} I \cap \pi_{e}(B_{0}) \neq \emptyset \quad \text{and} \quad I \in \calD, \end{displaymath}
where $\calD$ is the (standard) dyadic system in $\R$. We might also write \emph{$r$-tube} or \emph{$e$-tube} if the other parameter is clear from the context. We separately emphasise that tubes are always intersected with $B_{0}$; this is because we will be considering measures (satisfying \eqref{form39}) whose support is not contained in $B_{0}$, and we wish to avoid writing "$\nu(B_{0} \cap T)$" all the time.

Here are the main constants in the coming proof:
\begin{displaymath} 0 < \tau \ll \alpha < 1 \quad \text{and} \quad N \geq 1. \end{displaymath}
At the end of the day:
\begin{itemize}
\item $N$ is chosen first. It will depend on a counter assumption that the set in \eqref{form79} has Hausdorff dimension $\epsilon_{0} > 0$.
\item $\alpha$ can be chosen independently of $N$, and it will be chosen small enough to mitigate the evils caused by a large $N$.
\item $\tau$ can be chosen independently of both $\alpha$ and $N$, and it will be chosen small enough mitigate the evils caused by a small $\alpha$ and a large $N$.
\end{itemize}

It might be more illustrative to write "$\epsilon$" in place of "$\tau$" here, but $\epsilon$ is already reserved for the use in \eqref{form39}. We will adopt the following notation: $A \lessapprox_{p} B$ if there exist constants $C_{p},C_{\tau,p} \geq 1$ such that the following inequality holds for all $0 < \delta < 1$:
\begin{displaymath} A \leq C_{\tau,p}\delta^{-C_{p}\tau}B. \end{displaymath}
The notation $A \gtrapprox_{p} B$ means that $A \lessapprox_{p} B$, and the two-sided inequality $A \lessapprox_{p_{1}} B \lessapprox_{p_{2}} A$ will be abbreviated to $A \approx_{p_{1},p_{2}} B$. 

The constant $N$ is, in fact, the number (minus one) of elements in a fixed collection of dyadic rationals
\begin{displaymath} Q := \{q_{0},q_{1},\ldots,q_{N}\} \subset 2^{-\N} \cup \{0\} \quad \text{with} \quad 0 = q_{0} < q_{1} < \ldots < q_{N} = 1. \end{displaymath}
To be accurate, the final choice of the constant $\alpha > 0$ will depend on $Q$, and not only $N$. However, in the eventual application, $Q$ will have the form
\begin{displaymath} Q = \{q_{0},q_{1},\ldots,q_{N}\} = \{0,2^{-2N},2^{-N + 2},2^{-N + 3},\ldots,2^{-1},1\}, \end{displaymath}
so the statements "$\alpha$ will only depend on $N$" and "$\alpha$ will only depend on $Q$" are equivalent in the end. Finally, I also fix some (rapidly) increasing function
\begin{displaymath} f \colon \N \to \N \quad \text{with} \quad f(0) = 1. \end{displaymath}
The necessary rate of increase will be established during the proof below, rather implicitly, but $f(n + 1) > f(n)$ can always be chosen so that
\begin{displaymath} f(n + 1) \lesssim_{\alpha,Q} f(n). \end{displaymath}
In particular, the growth rate of $f$ \textbf{will not depend on $\tau$}, and hence any factors of the form $C_{\alpha,Q}f(n)\tau$, with $n \lesssim_{\alpha,Q} 1$, can eventually be made negligible by choosing $\tau > 0$ small in a manner depending only on $\alpha,Q$. During most of the proof (until the time we actually need to worry about choosing them), the quantities $\alpha$ and $Q$ will be regarded as "absolute", and I will abbreviate
\begin{equation}\label{lessapprox} A \lessapprox_{\alpha,Q} B \quad \text{to} \quad A \lessapprox B. \end{equation} 
To summarise, $A \lessapprox B$ means that $A \leq C\delta^{-C\tau} B$, where $C \geq 1$ is some constant depending on $\alpha$ and $Q$.

\subsection{Finding the good blow-up $\nu$}\label{tangentSection} Fix $e \in S^{1}$. This section contains an inductive construction -- or rather selection -- of a ball $B(x,R)$ with $x \in B_{0}$ and $0 < R \leq 1$, and a scale $0 < r < R$. In over-simplistic, but hopefully illustrative, terms, the selection will be done so that the quantity
\begin{displaymath} H_{e}(K \cap B(x,R),r) := \frac{\log N(\pi_{e}[K \cap B(x,R)],r)}{\log (R/r)} \end{displaymath}  
\textbf{cannot} be (substantially) decreased by replacing the triple $(x,R,r)$ by another triple $(x',R',r')$ satisfying $B(x',R') \subset B(x,R)$ and $1 \ll R'/r' \ll R/r$. A "substantial decrease" roughly means that 
\begin{equation}\label{form80} H_{e}(K \cap B(x',R'),r') \leq H_{e}(K \cap B(x,R),r) - \alpha \end{equation}
for some admissible triple $(x',R',r')$. Since $H_{e}(K \cap B(0,1),\delta) \leq 1$, we note that \eqref{form80} can only happen on $\lesssim 1/\alpha$ consecutive iterations. So, after $\lesssim 1/\alpha$ steps, one lands with a ball $B(x,R)$ and a scale $0 < r < R$ satisfying the opposite of \eqref{form80} for all admissible triples $(x',R',r')$. There were two great cheats in this discussion, which we briefly comment on. First, in reality, the quantity to be minimised is something more robust than $H_{e}(K \cap B(x,R),r)$. More accurately, we will (approximately) minimise a number "$s$" so that a large fraction of $\mu(B(x,R))$ can be covered by $\leq (R/r)^{s}$ tubes of width $r$ in direction perpendicular to $e$. The second cheat concerns this direction $e \in S^{1}$. As described above, the selection of $(x,R,r)$ is completely dependent on the initial choice of $e \in S^{1}$. However, in practice we need a uniform choice of $(x,R,r)$ for "$\sigma$-almost all $e \in S^{1}$". This would be easy if $\sigma = \delta_{e}$ for some fixed $e \in S^{1}$. In general, we can only achieve the desired uniformity inside a subset $S \subset S^{1}$ of measure $\sigma(S) \gtrapprox 1$. This is still good enough for practical applications.

\subsubsection{Some preliminaries}\label{preliminaries} Let $\delta > 0$ be a small "scale". At the end of the day, we will need to assume that $\delta > 0$ is small enough relative to $\alpha$ and $Q$, plus certain constants named arising from a quantitative counter assumption to Theorem \ref{mainTechnical}, see \eqref{data}. We also assume that if $g := \gcd Q$, then $\delta^{g^{-n}} \in 2^{-\N}$ for all $0 \leq n \lesssim 1/\alpha$. This is allowed, because we will not be making any claims for all $\delta > 0$; the initial dyadic scale $\delta = \delta_{0}$ simply has to be chosen sufficiently small (and, mainly for notational convenience, of the form $\delta \in 2^{-(g^{n_{\alpha}})\N}$).

We start by recording a frequently used corollary of the pigeonhole principle:
\begin{lemma}\label{pigeon} Let $C \geq 0$, $\delta \in 2^{-\N}$ and $e \in S^{1}$. Let $\nu \in \calM(\R^{2})$. Assume that there exists a collection $\calT$ of dyadic $(\delta,e)$-tubes whose union $E$ satisfies $\delta^{C} \leq \nu(E) \leq 1$. Then, there exists a subcollection $\calT' \subset \calT$, and a constant $0 \leq s \lesssim C$ such that

\begin{equation}\label{form82} \nu\left( \cup \calT' \right) \gtrsim \frac{\nu \left(E\right)}{C \log(1/\delta)} \quad \text{and} \quad \delta^{s} \leq \nu(E \cap T) \leq 2\delta^{s}, \, \: T,T' \in \calT'.\end{equation}
\end{lemma}

\begin{remark} Note that, as a corollary of \eqref{form82}, we have the following estimate for the cardinality of $\calT'$:
\begin{equation}\label{form83} \nu(E) \lessapprox \delta^{s} \cdot |\calT'| \leq \nu\left(E \right). \end{equation}
\end{remark}

\begin{proof}[Proof of Lemma \ref{pigeon}] For $j \in \Z$, let 
\begin{displaymath} \calT^{j} := \{T \in \calT : 2^{-j} \leq \nu(E \cap T) \leq 2^{-j + 1}\}. \end{displaymath}
Note that $\calT^{j} = \emptyset$ for $2^{-j} > \nu(E)$, and in particular for $j < 0$. Second, since $|\calT| \leq 2/\delta$, we have
\begin{displaymath} \nu \left(\bigcup_{2^{-j} < \delta^{C + 1}/16} \cup \calT^{j} \right) \leq \frac{2}{\delta} \sum_{2^{-j} < \delta^{C + 1}/16} 2^{-j + 1} \leq \frac{\delta^{C}}{2} \leq \frac{\nu(E)}{2}.  \end{displaymath}
Consequently, there exists $j \in \Z$ with $0 \leq j \lesssim C\log(1/\delta)$ such that \eqref{form82} holds for $s = \log_{\delta} 2^{-j}$ and $\calT' := \calT^{j}$. \end{proof}

\subsubsection{The induction hypotheses}\label{initialisation} Fix $e \in S^{1}$, and, to start an induction, write $\mu_{0} := \mu$ and $\delta_{0} := \delta$. Recall from \eqref{form39} that $0 < \mu(B_{0}) \leq 1$. Apply Lemma \ref{pigeon} to the collection $\calT$ of all dyadic $(\delta,e)$-tubes. Then, for $\tau > 0$ fixed, if $\delta > 0$ is small enough, there exists a number $0 \leq s_{0}(e) \lesssim 1$ and a collection of $\calT_{0}(e) := \calT' \subset \calT$ with the following properties:
\begin{itemize}
\item[(i)] $\delta_{0}^{-s_{0}(e) + \tau} \leq |\calT_{0}(e)| \leq \delta_{0}^{-s_{0}(e)}$, and
\item[(ii)] $\delta_{0}^{s_{0}(e)} \leq \mu_{0}(T) \leq 2\delta_{0}^{s_{0}(e)}$ for all $T \in \calT_{0}(e)$.
\end{itemize}
The same can be done for every $e \in S^{1}$, but naturally the quantities $s_{0}(e)$ and $|\calT_{0}(e)|$ vary. However, there exists a subset $S_{0} \subset S^{1}$ with $\sigma(S_{0}) \sim_{\alpha} 1$ and a number $s_{0} \in [0,d]$ such that
\begin{equation}\label{form10} |s_{0}(e) - s_{0}| \leq \frac{\alpha}{100}. \end{equation}
for all $e \in S_{0}$. There is nothing we wish -- or can -- do about the collections $\calT_{0}(e)$ varying with $e \in S_{0}$.

\begin{remark} I emphasise the obvious: $\tau > 0$ can be taken arbitrarily small here, just by adjusting the size of $\delta$. The choice of $\tau$ we make here will follow us, hidden in the $\lessapprox$-notation, until the very end of the proof. There we will finally decode the $\lessapprox$-notation to find a constant of the form $C_{\alpha,Q}\tau$. Then, as might be expected, we will need to make $C_{\alpha,Q}\tau$ less than some small number $\epsilon_{0} > 0$. We can indeed do so by returning back right here, and choosing $\delta > 0$ sufficiently small, depending on $\alpha,Q$, and $\epsilon_{0}$. \end{remark}

Next, we will attempt to decrease the number $s_{0}$ as much as we can, by either changing the scale $\delta_{0}$, or passing to a "rough tangent" -- or doing both. We assume inductively that we have already found the following objects for some $n \geq 0$:
\begin{itemize}
\item[(P1)\phantomsection\label{P1}] A scale $\delta_{0} \leq \delta_{n} \in 2^{-\N}$, which for $n \geq 1$ has the form $\delta_{n} = \delta_{n - 1}^{q_{j} - q_{i}}$ for some $q_{i} < q_{j}$,
\item[(P2)\phantomsection\label{P2}] A subset $S_{n} \subset S$ of measure $\sigma(S_{n}) \gtrsim_{\alpha,n,Q} \delta^{\Sigma(n)\tau}$, where $\Sigma(n) = \sum_{k = 0}^{n}f(n)$,
\item[(P3)\phantomsection\label{P3}] A number $s_{n} \leq s_{0}$,
\item[(P4)\phantomsection\label{P4}] For every $e \in S_{n}$, a number $s_{n}(e) \in [0,d]$ with $|s_{n}(e) - s_{n}| \leq \alpha/100$,
\item[(P5)\phantomsection\label{P5}] A measure $\mu_{n} \in \calM(\R^{2})$ of the form $\mu_{n} = \mu^{B(x_{n},r_{n})}$, where $x_{n} \in \R^{2}$ and $r_{n} \leq 1$, so in particular $\mu_{n}$ always satisfies \eqref{form39} by Lemma \ref{lemma2},
\item[(P6)\phantomsection\label{P6}] For every $e \in S_{n}$, a collection $\calT_{n}(e)$ of dyadic $(\delta_{n},e)$-tubes with
\begin{itemize}
\item[(i)] $|\calT_{n}(e)| \geq \delta_{n}^{-s_{n}(e) + f(n)\tau}$,
\item[(ii)] $\delta_{n}^{s_{n}(e)} \leq \mu_{n}(T) \leq 2\delta_{n}^{s_{n}(e)}$ for all $T \in \calT_{n}(e)$.
\end{itemize}
\end{itemize}

Clearly the conditions \nref{P1}-\nref{P6} are satisfied when $n = 0$, since $f(0) = 1$. We now explain, when to continue the induction -- and how -- and when to stop. 

\subsubsection{Bad balls in a fixed direction}\label{badBalls} Fix 
\begin{displaymath} q_{i},q_{j} \in Q \quad \text{with} \quad q_{i} < q_{j}, \end{displaymath}
and consider any ball $B := B(x,\delta_{n}^{q_{i}})$. For $e \in S_{n}$, the ball $B$ is called $(e,n)$-bad relative to the scale $\delta_{n}^{q_{j}}$ if there exists a number
\begin{equation}\label{form14} 0 \leq s_{n + 1}(e) \leq s_{n} - \alpha \end{equation}
and a disjoint collection of $(\delta_{0}^{q_{j} - q_{i}},e)$-tubes $\calT_{n + 1}(e)$ satisfying the properties in \nref{P6} for the index $n + 1$ with the choices
\begin{displaymath} \delta_{n + 1} = \delta_{n}^{q_{j} - q_{i}} \quad \text{and} \quad \mu_{n + 1} = \mu_{n}^{B}. \end{displaymath}
We spell out the conditions explicitly:
\begin{itemize}
\item[(i')] $|\calT_{n + 1}(e)| \geq (\delta_{n}^{q_{j} - q_{i}})^{-s_{n + 1}(e) + f(n + 1)\tau}$,
\item[(ii')] $(\delta_{n}^{q_{j} - q_{i}})^{s_{n + 1}(e)} \leq \mu_{n}^{B}(T) \leq 2(\delta_{n}^{q_{j} - q_{i}})^{s_{n + 1}(e)}$ for all $T \in \calT_{n + 1}(e)$.
\end{itemize}

\begin{remark} The badness of the ball $B$ depends on the choice of $e \in S_{n}$, so the choices of $\delta_{n + 1}$ and $\mu_{n + 1}$ above also depend on $e$. This is something we will deal with in a moment. It is also worth noting that the number $s_{n + 1}(e)$ and the tube family $\calT_{n + 1}(e)$ are far from unique. To emphasise this point, note that even the case $q_{0} = 0$ and $q_{1} = 1$ is allowed. Then $\delta_{n + 1} = \delta_{n}$, and the tube-collection $\calT_{n}(e)$ with roughly $\delta_{n}^{-s_{n}(e)}$ tubes simply gets replaced by another collection of $\calT_{n + 1}(e)$ of $\delta_{n}$-tubes. However, (ii') implies that $|\calT_{n + 1}(e)| \leq \delta_{n}^{-s_{n + 1}(e)}$, and now it follows from the condition \eqref{form14} that $\calT_{n + 1}(e)$ is significantly smaller than $\calT_{n}(e)$. \end{remark}

\subsubsection{Defining $\mu_{n + 1}$ and $\delta_{n + 1}$} Now, we know how to define the number $\delta_{n + 1} = \delta_{n + 1}(e)$, a measure $\mu_{n + 1} = \mu_{n + 1}(e)$ and the tube family $\calT_{n + 1}(e)$ \textbf{if} there exist $q_{i},q_{j} \in Q$ and and an $(e,n)$-bad ball $B(x,\delta^{q_{i}}_{n})$ relative to the scale $\delta_{n}^{q_{j}}$. Next we want to remove the dependence of $\mu_{n + 1}(e)$ and $\delta_{n + 1}(e)$ on the choice of $e \in S_{n}$.

\begin{definition}\label{bad} Let $q_{i},q_{j} \in Q$ with $q_{i} < q_{j}$. A vector $e \in S_{n}$ is called $(q_{i},q_{j})$-bad (a more precise term would be $(n,q_{i},q_{j})$-bad, but the index $n$ should be clear from the context in the sequel) if 
\begin{displaymath} \mu_{n}(\textbf{Bad}(q_{i},q_{j},e)) \geq \delta^{f(n + 1)\tau}, \end{displaymath}
where 
\begin{displaymath} \textbf{Bad}(q_{i},q_{j},e) := \{x \in B_{0} : B(x,\delta_{n}^{q_{i}}) \text{ is $(e,n)$-bad relative to the scale $\delta_{n}^{q_{j}}$}\}. \end{displaymath}
Finally, a vector $e \in S_{n}$ is called \emph{bad} (again, more accurately, \emph{$n$-bad}) if it is $(q_{i},q_{j})$-bad for some $q_{i},q_{j} \in Q$ with $q_{i} < q_{j}$.
\end{definition}

Here is the \emph{stopping condition} for the induction:
\begin{equation}\label{form11} \sigma(\{e \in S_{n} : e \text{ is bad}\}) < \tfrac{\sigma(S_{n})}{2}. \end{equation} 
In this case, we define
\begin{equation}\label{nu} \nu := \mu_{n}, \: \Delta := \delta_{n}, \: S := S_{n}, \: s := s_{n}, \quad \text{and} \quad s(e) := s_{n}(e) \text{ for } e \in S, \end{equation}
and the induction terminates. We will start examining this case in Section \ref{nuProperties}. For now, we discuss how to proceed with the induction if the stopping condition fails, that is,
\begin{displaymath} \sigma(\{e \in S_{n} : e \text{ is bad}\}) \geq \tfrac{\sigma(S_{n})}{2}. \end{displaymath} 
In this case, noting that $|Q \times Q| < \infty$, there exists a fixed pair $(q_{i},q_{j}) \in Q \times Q$ with $q_{i} < q_{j}$ such that
\begin{displaymath} \sigma(\{e \in S_{n} : e \text{ is $(q_{i},q_{j})$-bad}\}) \gtrsim_{Q} \sigma(S_{n}) \stackrel{\textup{\nref{P2}}}{\gtrsim_{\alpha,n,Q}} \delta^{\Sigma(n)\tau}. \end{displaymath}
Fix this pair $(q_{i},q_{j}) \in Q \times Q$. Then, by Fubini's theorem and the definition of $e$ being $(q_{i},q_{j})$-bad, we see that
\begin{align*} \int_{B_{0}} & \sigma(\{e \in S_{n} : x \in \textbf{Bad}(q_{i},q_{j},e)\}) \,d\mu_{n}(x)\\
& \geq \int_{\{e \in S_{n} : e \text{ is $(q_{i},q_{j})$-bad}\}} \mu_{n}(\textbf{Bad}(q_{i},q_{j},e)) \, d\sigma(e) \gtrsim_{\alpha,n,Q} \delta^{f(n + 1)\tau + \Sigma(n)\tau} = \delta^{\Sigma(n + 1)\tau}. \end{align*}
Since $\mu_{n}(B_{0}) \leq 1$, it follows that there exists $x_{0} \in B_{0}$ such that
\begin{equation}\label{form13} \quad \sigma(S_{n + 1}') \gtrsim_{\alpha,n + 1,Q} \delta^{\Sigma(n + 1)\tau}, \end{equation}
where
\begin{displaymath} S_{n + 1}' := \{e \in S_{n} : x_{0} \in \textbf{Bad}(q_{i},q_{j},e)\}. \end{displaymath}
Now, if $e \in S_{n + 1}'$, then $x_{0} \in \textbf{Bad}(q_{i},q_{j},e)$, which means by definition that that $B := B(x_{0},\delta_{n}^{q_{i}})$ is $(e,n)$-bad relative to the scale $\delta_{n}^{q_{j}}$. As explained in Section \ref{badBalls}, this allows us to define the objects
\begin{displaymath} \quad \mu_{n + 1} := \mu_{n}^{B}, \quad \delta_{n + 1} := \delta_{n}^{q_{j} - q_{i}}, \quad \text{and} \quad s_{n + 1}(e), \: \calT_{n + 1}(e) \text{ for } e \in S_{n + 1}. \end{displaymath}
In particular neither the measure $\mu_{n + 1}$ nor the scale $\delta_{n + 1}$ depend on the choice of $e \in S_{n + 1}'$. The condition \nref{P5} follows by the "chain rule" \eqref{chain}:
\begin{displaymath} \mu_{n + 1} = \mu_{n}^{B} = (\mu^{B_{n}})^{B} = \mu^{(T_{B} \circ T_{B_{n}})^{-1}(B_{0})}. \end{displaymath}
Here $B_{n + 1} := (T_{B} \circ T_{B_{n}})^{-1}(B_{0})$ is a ball of radius $r_{n + 1} := \delta_{n}^{q_{i}}r_{n}$. We have now managed to define all the objects mentioned in \nref{P1}-\nref{P6} -- for the index $n + 1$ -- except for the number $s_{n + 1}$. This is easily done: by the pigeonhole principle, there exists a number $s_{n + 1}$, and a further subset of $S_{n + 1} \subset S_{n + 1}'$ of measure $\sigma(S_{n + 1}) \sim_{\alpha} \sigma(S_{n + 1}')$ such that $|s_{n + 1}(e) - s_{n + 1}| \leq \alpha/100$ for all $e \in S_{n + 1}'$. Then \eqref{form13} continues to hold for $S_{n + 1}$ in place of $S_{n + 1}'$, and the induction may proceed.

\subsubsection{How soon is the stopping condition reached?} Recall from \eqref{form14} that $s_{n + 1}(e) \leq s_{n} - \alpha$ whenever the quantity $s_{n + 1}(e)$ is defined, and in particular for all $e \in S_{n + 1}$. So, recalling also that $|s_{n + 1} - s_{n + 1}(e)| \leq \alpha/100$ for all $e \in S_{n + 1}$, the numbers $s_{n}$ satisfy
\begin{displaymath} 0 \leq s_{n} \leq s_{0} - \alpha (n - 1)/2. \end{displaymath}
Since $s_{0} \lesssim 1$, this implies that the the induction can only run $\lesssim 1/\alpha$ steps before terminating. In particular, if $n$ is the index for which the induction terminates, and $S = S_{n}$ (as in \eqref{nu}), then
\begin{equation}\label{form41} \sigma(S) \gtrsim_{\alpha,Q} \delta^{\Sigma(n)\tau}\end{equation} 
with $n \lesssim 1/\alpha$. Here the size of $\Sigma(n)$ depends on $n \lesssim 1/\alpha$ and the growth rate of $f$ (which is further allowed to depend on $\alpha,Q$), so $\Sigma(n) \lesssim_{\alpha,Q} 1$. So, recalling our notational convention "$A \lessapprox B$" from \eqref{lessapprox}, we infer from \eqref{form41} that $\sigma(S) \gtrapprox 1$. 

\begin{remark} It is reasonable to ask: after all these blow-ups, what in the construction guarantees that $\nu$ is not e.g. the zero-measure? After all, the blow-up $\mu^{B}$ involved normalisation by $r(B)^{d}$ which could potentially be a lot larger than $\mu(B)$. However, condition \nref{P6} implies that
\begin{equation}\label{form42} \nu(B_{0}) \geq \Delta^{f(n)\tau} \gtrapprox 1. \end{equation} 
Moreover, since $\nu$ has the form $\mu^{B(x,r)}$ for some ball $B(x,r) \subset \R^{2}$, by \nref{P5}, we deduce that $\mu(B(x,r)) \gtrapprox r^{d}$. So, the definition of the "bad balls" was tailored so that the induction only ever moved along balls with reasonably large $\mu$ measure.
\end{remark}

\subsection{Projecting and slicing $\nu$}\label{nuProperties} We now assume that the induction has terminated at some index $n \lesssim 1/\alpha$, and the objects $\nu,\Delta,S,s$ and $s(e)$ have been defined as in \eqref{nu}. The letter $n$ will stand for this particular index for the rest of the paper. We may assume that $\Delta \in 2^{-\N}$, because by property \nref{P1} $\Delta$ has the form 
\begin{equation}\label{form40} \Delta = \delta^{\prod_{k = 1}^{n} (q_{j_{k}} - q_{i_{k}})} = \delta^{(\gcd Q)^{-n} \prod_{k = 1}^{n} (p_{j_{k}} - p_{i_{k}})} \end{equation}
for some $p_{i_{k}},q_{j_{k}} \in \N$, and we agreed in Section \ref{preliminaries} that $\delta^{(\gcd Q)^{-n}} \in 2^{-\N}$. Moreover, we will frequently need to assume that $\Delta$ is "very small" in a way depending on $\alpha,f(n),Q$. This can be done by selecting $\delta > 0$ small enough (depending on the same parameters), because by \eqref{form40} we have $\Delta \leq \delta^{(\gcd Q)^{-n}}$. 

We recall that the stopping condition \eqref{form11} has been reached at stage $n$, that is,
\begin{displaymath} \sigma(\{e \in S : e \text{ is bad}\}) < \tfrac{\sigma(S)}{2}. \end{displaymath}
By restricting $S$ to a subset of measure at least $\sigma(S)/2$, we may -- and will -- from now on assume that $S$ contains no bad vectors $e$. Let us spell out what this means. If $e \in S$, then $e$ is not $(q_{i},q_{j})$-bad for any $q_{i},q_{j} \in Q$ with $q_{i} < q_{j}$. Thus, for all such $q_{i},q_{j}$, we have
\begin{equation}\label{badMeasure} \nu(\textbf{Bad}(q_{i},q_{j},e)) < \Delta^{f(n + 1)\tau}. \end{equation}
In particular, by \eqref{form42}, for a $\nu$-majority of the points $x \in B_{0}$, the ball $B(x,\Delta^{q_{i}})$ is not bad relative to scale $\Delta^{q_{j}}$: this informally means that a large proportion of the $\nu$-measure in $B(x,\Delta^{q_{i}})$ cannot be captured by \textbf{much fewer} than $(\Delta^{q_{j} - q_{i}})^{-s}$ tubes perpendicular to $e$ and width $\Delta^{q_{j}}$. 

On the other hand, recalling the property \nref{P6} for the measure $\nu = \mu_{n}$, for $e \in S$, there exists $s(e) = s_{n}(e)$ with $|s(e) - s| \leq \alpha/100$, and a collection of dyadic $(\Delta,e)$-tubes $\calT(e) = \calT_{n}(e)$ such that
\begin{itemize}
\item[(i)] $|\calT| \geq \Delta^{-s(e) + f(n)\tau}$, and
\item[(ii)] $\Delta^{s(e)} \leq \nu(T) \leq 2\Delta^{s(e)}$ for all $T \in \calT(e)$.
\end{itemize}
It is crucial that $\nu\left(\cup \calT(e) \right) \geq \Delta^{f(n)\tau}$, which is substantially more than the measure of the points in $\textbf{Bad}(q_{i},q_{j},e)$ by \eqref{badMeasure}. While the precise information in (i) is often needed below, we also record that
\begin{equation}\label{form85} \Delta^{-s(e)} \lessapprox |\calT| \leq \Delta^{-s(e)}, \end{equation}
which follows immediately by combining (i)-(ii) and recalling that $\nu(B_{0}) \leq 1$.

The next aim is to show that $\nu$ is structured in the sense already discussed informally in Section \ref{mainOutline}(a)-(b). We plan to show that for any $e \in S$ there exists a set $K_{e} \subset B_{0}$ of measure $\nu(K_{e}) \gtrapprox 1$ such that $\pi_{e}(\nu|_{K_{e}})$ is "exact dimensional" with dimension $s$, and the restrictions of $\nu|_{(K_{e} \cap T)}$ to many $(\Delta,e)$-tubes $T$ look at least $(1 - s)$-dimensional at the scale $\Delta^{q_{1}} \gg \Delta$. 

Fix $e \in S$. This vector will remain fixed until Section \ref{dimensionSection}. So, until that, it will be convenient to assume that $e = (1,0)$, to call $e$-tubes simply \emph{tubes}, and to abbreviate $\calT_{n}(e) =: \calT$. Then, the tubes in $\calT$ are $\Delta$-neighbourhoods of vertical lines, intersected with $B_{0}$. We now start building the good subset $K_{e}$ mentioned above: it will satisfy $K_{e} \subset \cup \calT$.

\subsubsection{Non-concentration of $\nu$ in $\Delta$-tubes} Recall the dyadic rationals $Q = \{q_{0},\ldots,q_{N}\}$ with $q_{0} = 0$ and $q_{N} = 1$. Write
\begin{displaymath} q := q_{1} \in (0,1) \end{displaymath}
for the smallest non-zero rational in $Q$.
\begin{lemma}\label{NCLemma} The following holds if $f(n + 1)$ is sufficiently large and $\Delta,\tau > 0$ are sufficiently small in terms of $\alpha,f(n),Q$. There exist ("good") subsets $\calT_{G} \subset \calT$ and $K_{G} \subset \cup \calT_{G} \cap (\spt \nu)$ such that
\begin{itemize}
\item[\textup{(NC1)}\phantomsection\label{NC1}] $|\calT_{G}| \approx |\calT|$ and $\nu(K_{G} \cap T) \approx \nu(T)$ for all $T \in \calT_{G}$,
\item[\textup{(NC2)}\phantomsection\label{NC2}] for $T \in \calT_{G}$ and all $x \in \R^{2}$,
\begin{displaymath} \frac{\nu([K_{G} \cap T] \cap B(x,\Delta^{q}))}{\nu(T)} \lessapprox (\Delta^{q})^{d - s(e) - 4\alpha/q}. \end{displaymath}
\end{itemize}
\end{lemma}

\begin{proof} Cover each tube $T \in \calT$ by a family $\calR_{T}'$ of dyadic rectangles of the form $I \times J$, where $I = \pi_{1}(T)$ (hence $\ell(I) = \Delta$) and $\ell(J) = \Delta^{q}$. By the upper $d$-regularity of $\nu$ (recall \eqref{form39} and \nref{P5}),
\begin{displaymath} \nu(R) \lesssim \Delta^{qd}, \qquad R \in \calR_{T}'. \end{displaymath} 
Now, fixing $T \in \calT$ and noting that $\Delta^{s(e)} \leq \nu(T) \leq 1$, we may use a pigeonholing argument similar to the one used in Lemma \ref{pigeon} to choose a dyadic number $0 \leq m_{T} \leq \Delta^{q d}$ and a subset $\calR_{T} \subset \calR_{T}'$ such that
\begin{itemize}
\item $\nu(R) \sim m_{T}$ for all $R \in \calR_{T}$, and
\item $\nu(\cup \calR_{T}) \approx \nu(T) \sim \Delta^{s(e)}$.
\end{itemize}
(Alternatively, one could apply Lemma \ref{pigeon} directly to the measure $\nu|_{T}$ and the family of all $(\Delta^{q},(0,1))$-tubes intersecting $T$.) We write
\begin{displaymath} T_{G} := \bigcup_{R \in \calR_{T}} R. \end{displaymath}
Next, we run one more pigeonholing argument to make the number $m_{T}$ uniform among the tubes $T \in \calT$: there exists a number $0 \leq m \leq \Delta^{q d}$, and a collection $\calT_{G} \subset \calT$ of cardinality
\begin{equation}\label{form84} |\calT_{G}| \approx |\calT| \quad \text{such that} \quad m_{T} \sim m \, \text{ for all } \, T \in \calT_{G}. \end{equation}
Write
\begin{displaymath} K_{G} := \bigcup_{T \in \calT_{G}} T_{G} = \bigcup_{T \in \calT_{G}} \bigcup_{R \in \calR_{T}} R = \bigcup_{R \in \calR_{G}} R, \end{displaymath} 
where $\calR_{G}$ stands for the union of the collections $\calR_{T}$ with $T \in \calT_{G}$; thus $\nu(R) \sim m$ for all $R \in \calR_{G}$. An immediate consequence of the choices of $\calT_{G}$ and $\calR_{T}$ is that $\nu(K_{G} \cap T) = \nu(T_{G}) \gtrapprox \nu(T) \geq \Delta^{s(e)}$ for all $T \in \calT_{G}$, so now all the points in \nref{NC1} have been addressed. 

It remains to prove \nref{NC2}, and we will do this by showing that
\begin{equation}\label{form87} m \leq \Delta^{q(d - s(e)) + s(e) - 4\alpha}, \end{equation}
if $\Delta > 0$ is small enough, and the function $f$ is sufficiently rapidly increasing. Note that \eqref{form87} implies \nref{NC2}, because $[K_{G} \cap T] \cap B(x,\Delta^{q})$ can always be covered by $\lesssim 1$ rectangles in $\calR$, and $\nu(T) \approx \Delta^{s(e)}$ for $T \in \calT_{G}$. 

The idea behind the proof of \eqref{form87} is the following. The number "$m$" represents the $\nu$-measure of a "typical" vertical rectangle of dimensions $\Delta \times \Delta^{q}$. Using the $d$-quasiregularity of $\nu$, we can calculate the number of such "typical" rectangles intersecting a "typical" ball of radius $\Delta^{q}$ (of $\nu$-measure $\approx \Delta^{dq}$). Then, if $m$ violates \eqref{form87}, it follows that a large part of the $\nu$-measure in such a "typical" $\Delta^{q}$-ball is contained in "rather few" vertical $\Delta$-tubes. Now, a "typical" $\Delta^{q}$-ball is \textbf{not} an $e$-bad ball relative to scale $\Delta$, so in fact only a very small fraction of the $\nu$-measure in such a ball can be covered by "rather few" $\Delta$-tubes. This eventually gives the upper bound \eqref{form87}.

We turn to the details. Start by noting that 
\begin{equation}\label{form22} \nu(K_{G}) = \sum_{T \in \calT_{G}} \nu(K_{G} \cap T) \stackrel{\textup{\nref{NC1}}}{\gtrapprox} |\calT_{G}| \cdot \Delta^{s(e)} \stackrel{\eqref{form84}}{\gtrapprox} \Delta^{f(n)\tau}. \end{equation}
Recalling the definition of the "$\gtrapprox$" notation from \eqref{lessapprox}, this means that $\nu(K_{G}) \gtrsim_{\alpha,Q} \Delta^{C_{\alpha,Q}\tau + f(n)\tau}$ for some constant $C_{\alpha,Q} \geq 1$. With this notation, write
\begin{displaymath} C  := 4[C_{\alpha,Q} + f(n)] \sim_{\alpha,Q} 1. \end{displaymath}
Then, by Lemma \ref{lemma1}, the set
\begin{displaymath} E := E_{\Delta,C\tau}(\Delta^{q}) = \{x \in K_{G} : \nu(K_{G} \cap B(x,\Delta^{q})) < \Delta^{C\tau}(\Delta^{q})^{d}\} \end{displaymath} 
has 
\begin{displaymath} \nu(E) \leq C_{\tau}\Delta^{(C/2)\tau} = C_{\tau}\Delta^{2[C_{\alpha,Q} + f(n)]} < \frac{\nu(K_{G})}{4}, \end{displaymath}
assuming that $\Delta > 0$ is sufficiently small in a way depending on $\alpha,Q$ (and $\tau$, which depends here and will always depend only on $\alpha,Q$). Note also that
\begin{displaymath} \nu(\textbf{Bad}(q,1,e)) < \Delta^{f(n + 1)\tau} < \frac{\nu(K_{G})}{4} \end{displaymath}
by \eqref{badMeasure} if $f(n + 1) \geq C$, and again $\Delta > 0$ is sufficiently small. Now, we infer that there exists a point
\begin{displaymath} x_{0} \in K_{G} \, \setminus \, [\textbf{Bad}(q,1,e) \cup E] \end{displaymath}
which then satisfies
\begin{equation}\label{form23} \nu(K_{G} \cap B(x_{0},\Delta^{q})) \geq \Delta^{q d + C\tau} \end{equation}
by definition of $x_{0} \notin E$. Write
\begin{displaymath} B := B(x_{0},\Delta^{q}). \end{displaymath}
Now, recall that the set $K_{G}$ is a union of the disjoint rectangles $R \in \calR_{G}$ of dimensions $\Delta \times \Delta^{q}$, each satisfying $\nu(R) \sim m$. If $T \in \calT_{G}$ is any tube such that $K_{G} \cap B \cap T \neq \emptyset$, then $K_{G} \cap B$ meets one of these rectangles, say $R_{B,T} \in \calR_{T}$, and evidently $R_{B,T} \subset 10B$. Since the tubes $T \in \calT_{G}$ are disjoint, the corresponding rectangles $R_{B,T}$ are also disjoint, and we find that
\begin{equation}\label{form24} |\{T \in \calT_{G} : K_{G} \cap B \cap T \neq \emptyset\}| \lesssim \frac{\nu(10B)}{m} \lesssim \frac{\Delta^{q d}}{m}. \end{equation}
Consider now the blow-up $\nu^{B}$. According to \eqref{form23}-\eqref{form24}, and noting that $K_{G} \subset \cup \calT_{G}$, there exists a subset of $\nu^{B}$-measure $\geq \Delta^{C\tau}$ which can be covered by $\lesssim \Delta^{q d}/m$ dyadic $\Delta^{1- q}$-tubes. We denote these tubes by $\calT_{n + 1}'$. We claim that
\begin{equation}\label{form26} \frac{\Delta^{q d}}{m} \geq (\Delta^{1 - q})^{-s(e) + 4\alpha}, \end{equation}
which implies \eqref{form87} after re-arranging terms. Assume to the contrary that 
\begin{equation}\label{form27} |\calT_{n + 1}'| \lesssim \frac{\Delta^{qd}}{m} < (\Delta^{1 - q})^{-s(e) + 4\alpha}. \end{equation}
Then, since $\nu^{B}\left( \cup \calT_{n + 1}' \right) \geq \Delta^{C\tau}$, we can apply Lemma \ref{pigeon} to find a constant $0 \leq s_{n + 1}(e) \lesssim 1$ and a subcollection $\calT_{n + 1} \subset \calT_{n + 1}'$ such that
\begin{equation}\label{form30} (\Delta^{1 - q})^{s_{n + 1}(e)} \leq \nu^{B}(T) \leq 2(\Delta^{1 - q})^{s_{n + 1}(e)}, \qquad T \in \calT_{n + 1}, \end{equation}
and 
\begin{displaymath} \nu^{B} \left( \cup \calT_{n + 1} \right) \gtrapprox \Delta^{C\tau}. \end{displaymath}
From \eqref{form27} and \eqref{form83}, we now infer that
\begin{equation}\label{form30b} (\Delta^{1 - q})^{\tfrac{C\tau}{1 - q}} = \Delta^{C\tau} \lessapprox  |\calT_{n + 1}| \cdot (\Delta^{1 - q})^{s_{n + 1}(e)} \lesssim (\Delta^{1 - q})^{-s(e) + 4\alpha + s_{n + 1}(e)}. \end{equation} 
Now, if $\tau > 0$ is sufficiently small in a way depending only on $C \sim_{\alpha,Q} 1$, and noting that $1 - q \geq 1/2$, the exponent $C\tau/(1 - q)$ on the left hand side can be taken less than $\alpha$. Thus,
\begin{equation}\label{form28} (\Delta^{1 - q})^{s_{n + 1}(e) - s(e)} \gtrapprox (\Delta^{1 - q})^{-3\alpha} \quad \Longrightarrow \quad s_{n + 1}(e) \leq s(e) - 2\alpha \leq s - \alpha, \end{equation}  
assuming that $\Delta,\tau > 0$ are sufficiently small, and recalling that $|s(e) - s| \leq \alpha/100$. Moreover, from the leftmost inequality of \eqref{form30b}, and taking $f(n + 1)$ a sufficient amount larger than $C/(1 - q)$ (depending on the implicit constants in \eqref{form30b}), we have
\begin{equation}\label{form29} |\calT_{n + 1}| \gtrapprox (\Delta^{1 - q})^{-s_{n + 1}(e) + \tfrac{C\tau}{1 - q}} \quad \Longrightarrow |\calT_{n + 1}| \geq (\Delta^{1 - q})^{-s_{n + 1}(e) + f(n + 1)\tau},  \end{equation} 
if $\Delta > 0$ is small enough. But the estimates \eqref{form30} and \eqref{form28}-\eqref{form29} combined now literally say that the ball $B = B(x_{0},\Delta^{q})$ is $e$-bad relative to the scale $\Delta = \Delta^{1}$, see Section \ref{badBalls}, and hence $x_{0} \in \textbf{Bad}(q,1,e)$, contradicting the choice of $x_{0}$. This completes the proof of \eqref{form26}, and hence the proof of \nref{NC2} and Lemma \ref{NCLemma} -- except that Lemma \ref{NCLemma} also claims that $K_{G} \subset \spt \nu$. However, this can be achieved by intersecting $K_{G}$, as above, with $\spt \nu$ without affecting either \nref{NC1} or \nref{NC2}. \end{proof}

%%%%%%%%%%%%%%

\begin{comment}

\subsubsection{Summary} We recap the achievements so far. We have found a subset $K_{0}^{e}$ of measure $\nu(K_{0}^{e}) \gtrapprox \Delta^{f(n)\tau}$, which can be written as a disjoint union of (parallel) $\Delta \times \Delta^{q}$-rectangles in a collection $\calR$. Moreover, all the rectangles in $\calR$ can be covered by a collection of $(\Delta,e)$-tubes called $\calT(e)'$ with the properties that
\begin{equation}\label{form31} \card \calT(e)' \gtrapprox \Delta^{-s(e) + f(n)\tau} \quad \text{and} \quad \nu(T) \sim \Delta^{s(e)}, \: T \in \calT(e)'. \end{equation} 
The total $\nu$-measure of the rectangles in $\calR$ contained in any given tube $T \in \calT(e)'$ is also $\approx \Delta^{s(e)}$. Finally, each rectangle $R \in \calR$ satisfies $\nu(R) \sim m$ with
\begin{equation}\label{form74} \frac{m}{\Delta^{s(e)}} \leq \Delta^{q(d - s(e)) - 4\alpha(1 - q)}. \end{equation}

\end{comment}

%%%%%%%%%%%%%%

\subsubsection{Branching of the tubes in $\calT$} Recall the tube family $\calT_{G} \subset \calT$ constructed in Lemma \ref{NCLemma}. We write
\begin{equation}\label{form86} \calT^{0} := \calT_{G}, \quad t_{0} := s(e), \quad \text{and} \quad K_{e}^{0} := K_{G} \subset \spt \nu \cap B_{0}, \end{equation}
then $|\calT^{0}| \approx \Delta^{-t_{0}}$ and $\nu(K_{e}^{0} \cap T) \approx \Delta^{t_{0}}$ for all $T \in \calT^{0}$  by \eqref{form85} and \nref{NC1}. 

As in the previous section, view the vector $e \in S$ as "fixed" -- that is, we omit it from the notation as much as we can. Now, we claim inductively that if the parameter $\tau > 0$ is taken sufficiently small, depending on $\alpha$ and $Q$, then for all $0 \leq k \leq N - 1$ there exists a number $t_{k}$ and a collection of dyadic $\Delta^{q_{N - k}}$-tubes $\calT^{k}$ with the following properties:
\begin{itemize}
\item[(B1)\phantomsection\label{B1}] $|t_{k} - s(e)| \leq 2\alpha/q_{1}$,
\item[(B2)\phantomsection\label{B2}] $|\calT^{k}| \approx (\Delta^{q_{N - k}})^{-t_{k}}$ and $\nu(K_{e}^{k} \cap T) \approx (\Delta^{q_{N - k}})^{t_{k}}$ for $T \in \calT^{k}$, where
\begin{displaymath} K_{e}^{k} := K_{e}^{k - 1} \cap \left( \cup \calT^{k} \right), \qquad 1 \leq k \leq N - 1. \end{displaymath}
\item[(B3)\phantomsection\label{B3}] If $1 \leq k \leq N - 1$ and $T^{k} \in \calT^{k}$, then 
\begin{displaymath} |\calT^{k - 1}(T^{k})| \approx \frac{|\calT^{k - 1}|}{|\calT^{k}|} \approx \Delta^{-t_{k - 1}q_{N - k + 1} + t_{k}q_{N - k}}, \end{displaymath}
where $\calT^{k - 1}(T^{k}) := \{T^{k - 1} \in \calT^{k - 1} : T^{k - 1} \subset T^{k}\}$.
\end{itemize}

\begin{remark} Recall \eqref{lessapprox}: the notation $A \lessapprox B$ means that $A \leq C \Delta^{-C\tau}B$, where the constant $C$ is allowed to depend on $\alpha$ and $Q$. In particular, once we start proving \nref{B1}-\nref{B3} by induction on $k \in \{0,\ldots,N - 1\}$, the implicit constants in the "$\approx$" notation are allowed to get worse as $k$ increases -- because $k$ will only increase up to $N - 1 \lesssim_{Q} 1$. \end{remark}

Noting that $q_{N - 0} = 1$, the choices made in \eqref{form86} evidently satisfy \nref{B1}-\nref{B3}. So, we may assume that the number $t_{k - 1}$ and the collection $\calT^{k - 1}$ have already been found for some $1 \leq k \leq N - 1$. To proceed, we use the pigeonhole principle: there exist a number $t_{k} \in [0,1]$ and collection $\calT^{k}$ of dyadic $\Delta^{q_{N - k}}$-tubes with $|\calT^{k}| \approx (\Delta^{q_{N - k}})^{-t_{k}}$ such that
\begin{equation}\label{form32} |\calT^{k - 1}(T^{k})| \approx \frac{|\calT^{k - 1}|}{|\calT^{k}|} \approx \Delta^{-t_{k - 1}q_{N - k + 1} + t_{k}q_{N - k}}, \qquad T^{k} \in \calT^{k}. \end{equation} 
The second equation in \eqref{form32} used the inductive hypothesis \nref{B2} on $\calT^{k - 1}$. We have now found the objects $t_{k},\calT^{k}$, and established \nref{B3} and the first part of \nref{B2}. The second part of \nref{B2} follows from \eqref{form32}, and the inductive hypothesis \nref{B2} for the tube family $\calT^{k - 1}$: fixing $T \in \calT^{k}$, and defining $K_{e}^{k}$ as in \nref{B2}, we find that
\begin{align*} \nu(K_{e}^{k} \cap T) & = \nu(K_{e}^{k - 1} \cap T)\\
& = \sum_{T' \in \calT^{k - 1}(T)} \nu(K^{k - 1}_{e} \cap T')\\
& \stackrel{\textup{\nref{B2}}}{\approx} |\calT^{k - 1}(T)| \cdot (\Delta^{q_{N - k + 1}})^{t_{k - 1}}\\
& \stackrel{\eqref{form32}}{\approx} (\Delta^{q_{N - k}})^{t_{k}}. \end{align*}
It remains to establish \nref{B1}, which states that the "branching" is roughly constant for all levels $0 \leq k \leq N - 1$. The proof of \nref{B1} bears close similarity to the proof of property \nref{NC2} in Lemma \ref{NCLemma}: if \nref{B1} failed, we would end up finding some bad balls where none should exist. We will prove separately that
\begin{itemize}
\item[(a)] $t_{k} \geq s(e) - 2\alpha$,
\item[(b)] $t_{k} \leq s(e) + 2\alpha/q_{1}$. 
\end{itemize}
We start with the slightly easier task (a), and make a counter assumption: 
\begin{equation}\label{form91} t_{k} < s(e) - 2\alpha. \end{equation}
Write $q := q_{N - k}$. Then, by \nref{B2}, 
\begin{displaymath} |\calT^{k}| \gtrsim (\Delta^{q})^{-t_{k} + C\tau} \quad \text{and} \quad \nu\left( K_{e}^{k} \cap T \right) \gtrsim (\Delta^{q})^{t_{k} + C\tau} \text{ for } T \in \calT^{k}, \end{displaymath}
where $C \geq 1$ is a constant depending only on $\alpha$ and $Q$. In particular, $\nu(K_{e}^{k}) \gtrsim \Delta^{2C\tau}$. Recall from \eqref{badMeasure} that
\begin{displaymath} \nu(\textbf{Bad}(0,q,e)) < \Delta^{f(n + 1)\tau}. \end{displaymath}
Now, we infer from Lemma \ref{lemma1} that the $\nu$ measure of the set
\begin{displaymath} E = E_{\Delta,6C\tau}(1) = \{x \in K_{e}^{k} : \nu(K_{e}^{k} \cap B(x,1)) < \Delta^{6C\tau}\} \end{displaymath}
is bounded by $\nu(E) \leq C_{\tau}\Delta^{3C\tau}$, and in particular $\nu(E) \leq \nu(K_{e}^{k})/4$ if $\Delta > 0$ is small enough. If the function $f$ is rapidly increasing enough, we also have $2C < f(n + 1)$, and hence we may find a point 
\begin{displaymath} x_{0} \in K_{e}^{k} \, \setminus \, \textbf{Bad}(0,q,e) \end{displaymath}
with
\begin{displaymath} \nu \left(K_{e}^{k} \cap B(x_{0},1) \right) \geq \Delta^{6C\tau}. \end{displaymath}
Write $B := B(x_{0},1)$, and assume without loss of generality here that $x_{0} = 0$, so that $\nu^{B} = \nu$ (otherwise some of the tubes below need to be translated by $x_{0}$). Applying Lemma \ref{pigeon} and its corollary \eqref{form83} to the family $\calT = \calT^{k}$ (whose union covers $K_{e}^{k}$), we find a number $0 \leq s_{n + 1}(e) \lesssim 1$ and a subset $\calT_{n + 1} \subset \calT^{k}$ such that
\begin{equation}\label{form88} |\calT_{n + 1}| \gtrapprox (\Delta^{q})^{-s_{n + 1}(e) + 6C\tau/q} \end{equation}
and
\begin{equation}\label{form90} (\Delta^{q})^{s_{n + 1}(e)} \leq \nu^{B}(T) \leq 2(\Delta^{q})^{s_{n + 1}(e)} \text{ for } T \in \calT_{n + 1}. \end{equation}
Since $\calT_{n + 1} \subset \calT^{k}$, and $|\calT^{k}| \lesssim (\Delta^{q})^{-t_{k} - C'\tau}$ by \nref{B2}, we infer from \eqref{form88} and our counter assumption \eqref{form91} that
\begin{equation}\label{form89} s_{n + 1}(e) \leq t_{k} + 7(C + C')\tau/q \leq s(e) - 2\alpha + 7C\tau/q \leq s - \alpha, \end{equation} 
if $\tau > 0$ is sufficiently small in terms of $\alpha$ and the constants $C,C' \sim_{\alpha,Q} 1$. Now, assuming that $f(n + 1)$ is larger than $6C/q$ plus the implicit constants hidden in \eqref{form88}, and then taking $\Delta > 0$ is small enough, \eqref{form88} gives
\begin{displaymath} |\calT_{n + 1}| \geq (\Delta^{q})^{-s_{n + 1}(e) + f(n + 1)\tau}. \end{displaymath}
This combined with \eqref{form90}-\eqref{form89} means that $B$ is an $e$-bad ball relative to the scale $\Delta^{q}$, recall Section \ref{badBalls}. Hence $x_{0} \in \textbf{Bad}(0,q,e)$. This contradiction proves that $t_{k} \geq s(e) - 2\alpha$.

Next, we undertake the task of verifying (b). Assume for contradiction that 
\begin{equation}\label{form33} t := t_{k} > s(e) + \frac{2\alpha}{q_{1}} \quad \Longrightarrow \quad t - s(e) > \frac{2(1 - q)\alpha}{q}, \end{equation}
where we have again written $q := q_{N - k} \geq q_{1}$. Iterating \nref{B3}, and setting $\calT^{0}(T^{k}) := \{T^{0} \in \calT^{0} : T^{0} \subset T^{k}\}$, we have
\begin{equation}\label{form34} \card \calT^{0}(T^{k}) \approx \prod_{j = 1}^{k} \frac{\card \calT_{j - 1}}{\card \calT_{j}} \approx \prod_{j = 1}^{k} \Delta^{-t_{j - 1}q_{N - j + 1} + t_{j}q_{N - j}} = \Delta^{-s(e) + tq} \end{equation} 
for all $T^{k} \in \calT^{k}$. Note here that
\begin{equation}\label{form35} -s(e) + tq = (1 - q) \cdot \left(-s(e) + \frac{q(t - s(e))}{1 - q} \right) > (1 - q) \cdot (-s(e) + 2\alpha) \end{equation}
by \eqref{form33}. This time, we use \eqref{badMeasure} in the form
\begin{displaymath} \nu(\textbf{Bad}(q,1,e)) < \Delta^{f(n + 1)\tau}. \end{displaymath}
On the other hand, by \nref{B2}, $\nu(K^{k}_{e}) \gtrsim \Delta^{C\tau}$ for some constant $C \geq 1$ depending only on $\alpha,Q$. So, we may infer from Lemma \ref{lemma1} that
\begin{displaymath} \nu(\{x \in K_{k}^{e} : \nu(K_{e}^{k} \cap B(x,\Delta^{q})) \leq \Delta^{qd + 4C\tau}\}) = \nu(E_{\Delta,4C\tau}(\Delta^{q})) \lesssim_{\tau} \Delta^{2C\tau}. \end{displaymath}
Consequently, if $f(n + 1) > 2C$ and $\Delta > 0$ is small enough, we may find a point
\begin{equation}\label{form94} x_{0} \in K^{k}_{e} \, \setminus \, \textbf{Bad}(q,1,e) \end{equation}
satisfying
\begin{equation}\label{form36} \nu(K_{e}^{k} \cap B(x_{0},\Delta^{q})) \geq \Delta^{qd + 4C\tau}. \end{equation} 
Write $B := B(x_{0},\Delta^{q})$, and observe that the set $K_{e}^{k} \cap B$ can be covered by at most three $\Delta^{q}$-tubes in the collection $\calT^{k}$, say $T^{k_{1}},T^{k_{2}},T^{k_{3}}$. Consequently
\begin{displaymath} K_{e}^{k} \cap B \subset \left( \cup \calT^{0}(T^{k_{1}}) \right) \cup \left( \cup \calT^{0}(T^{k_{2}}) \right) \cup \left( \cup \calT^{0}(T^{k_{3}}) \right). \end{displaymath}
Then, combining \eqref{form34}-\eqref{form35}, we infer that $K_{e}^{k} \cap B$ can be covered by a total of
\begin{equation}\label{form37} \lesssim (\Delta^{1 - q})^{-s(e) + 2\alpha - C'\tau} \end{equation}
$\Delta$-tubes in $\calT^{0}$, say $\calT_{B}$, where $C' \sim_{\alpha,Q} 1$. Recalling \eqref{form36}, this means that there exists a collection $\calT_{n + 1}'$ of $\Delta^{1 - q}$-tubes, namely the images of the tubes in $\calT_{B}$ under the homothety $B \to B_{0}$, satisfying the cardinality bound \eqref{form37}, such that
\begin{equation}\label{form38} \nu^{B}\left( \cup \calT_{n + 1}' \right) = \nu^{B(x_{0},\Delta^{q})}\left( \cup \calT_{n + 1}' \right) \geq \Delta^{4C\tau}. \end{equation}
This implies that $B$ is an $e$-bad ball relative to the scale $\Delta$ by an argument we have already seen a few times. Namely, combining \eqref{form37}-\eqref{form38} and using Lemma \ref{pigeon} (and its corollary \eqref{form83}), we can find a number $0 \leq s_{n + 1}(e) \lesssim 1$ and a subcollection $\calT_{n + 1} \subset \calT_{n + 1}'$ such that
\begin{equation}\label{form92} (\Delta^{1 - q})^{s_{n + 1}(e)} \leq \nu^{B}(T) \leq 2(\Delta^{1 - q})^{s_{n + 1}(e)}, \qquad T \in \calT_{n + 1}, \end{equation}
and 
\begin{equation}\label{form93} |\calT_{n + 1}| \geq (\Delta^{1 - q})^{-s_{n + 1}(e) + 5C\tau}. \end{equation}
Since $\calT_{n + 1} \subset \calT_{n + 1}'$, we moreover have from \eqref{form37} that
\begin{displaymath} (\Delta^{1 - q})^{-s_{n + 1}(e) + 5C\tau} \leq |\calT_{n + 1}| \leq |\calT_{n + 1}'| \lesssim (\Delta^{1 - q})^{-s(e) + 2\alpha - C'\tau}, \end{displaymath}
whence $s_{n + 1}(e) \leq s(e) - 2\alpha + C''\tau \leq s - \alpha$, assuming $\tau > 0$ small enough. If $f(n + 1) \geq 5C\tau$, a combination of \eqref{form92}-\eqref{form93} now means that $B$ is an $e$-bad ball relative to the scale $\Delta$, recall Section \ref{badBalls}, and hence $x_{0} \in \textbf{Bad}(q,1,e)$. This contradicts the choice of $x_{0}$ in \eqref{form94} and completes the proof of (b), namely that $t_{k} \leq s(e) + 2\alpha/q_{1}$. The proof of the properties \nref{B1}-\nref{B3} is also now complete.

We now set
\begin{equation}\label{form44} K_{e} := K_{e}^{N - 1}, \end{equation}
so in particular $K_{e} \subset K_{G}$ (from Lemma \ref{NCLemma}) and $\nu(K_{e}) \gtrapprox 1$ for all $e \in S$, using \nref{B2}.

\section{Proof of the main theorem}\label{dimensionSection}

\subsection{Preliminaries} The constructions from the previous section only assumed that $\mu \in \calM(B_{0})$ was a $d$-quasiregular measure, and that $\sigma$ was an arbitrary Borel probability measure on $S^{1}$. We now specialise the considerations to prove Theorem \ref{mainTechnical}, whose statement is repeated below:
\begin{thm} Let $0 \leq d \leq 2$, and let $\mu \in \calM(B_{0})$ be a $d$-quasiregular measure, and write $K := \spt \mu$. Then
\begin{displaymath} \Hd \{e \in S^{1} : \Ad \pi_{e}(K) < \min\{d,1\}\} = 0. \end{displaymath}
\end{thm}

\subsubsection{Some standard reductions}\label{s:reductions} First, by the countable stability of Hausdorff dimension, it suffices to fix a number $0 < D < \min\{d,1\}$ and prove that $\Hd S_{0} = 0$, where 
\begin{displaymath} S_{0} := \{e \in S^{1} : \Ad \pi_{e}(K) < D\}. \end{displaymath}
We make a counter assumption: $\calH^{\epsilon_{0}}_{\infty}(S_{0}) > 0$ for some $\epsilon_{0} > 0$. Then we fix a scale $\delta > 0$, which needs to be assumed small in a manner depending, eventually, on 
\begin{equation}\label{data} \epsilon_{0}, \: d - D,\: 1 - D,\: \calH^{\epsilon_{0}}_{\infty}(S_{0}), \quad \text{and} \quad C_{E} \end{equation}
in the upcoming estimate in \eqref{form43}. We stated in Section \ref{preliminaries} that $\delta$ also needs to be chosen small enough relative to $\alpha$ and $Q$, but these parameters will only depend on the constants in \eqref{data}. Then, we choose, using Frostman's lemma, a Borel probability measure $\sigma \in \calM(S_{0})$ satisfying 
\begin{equation}\label{frostman} \sigma(B(e,r)) \leq C_{\sigma}r^{\epsilon_{0}}, \qquad e \in S^{1}, \: \delta < r \leq 1, \end{equation} 
for some constant $C_{\sigma} \geq 1$ depending only on $\calH^{\epsilon_{0}}_{\infty}(S_{0})$. The reader may check that $S_{0}$ is Borel, and apply the standard version of Frostman's lemma. But since we only need \eqref{frostman} for the scales $\delta < r \leq 1$, no measurability is really needed: the fact that $\calH^{\epsilon_{0}}_{\infty}(S_{0}) = c > 0$ can be used to find a \emph{$(\delta,\epsilon_{0})$-set} $P \subset S_{0}$ of cardinality $|P| \sim_{c} \delta^{-\epsilon_{0}}$, see the proof of \cite[Proposition A.1]{FaO}, and then the choice $\sigma = |P|^{-1}\calH^{0}|_{P} \in \calM(S_{0})$ satisfies \eqref{frostman}. We also need to quantify the fact that $\Ad \pi_{e}(K) < D$ for $e \in S_{0}$. In fact, we may assume that the following inequality holds for $\sigma$ almost all $e \in S_{0}$, and for all $x \in \R$ and $0 < r < R < \infty$:
\begin{equation}\label{form43} N(\pi_{e}(K) \cap B(x,R),r) \leq C_{E}\left(\frac{R}{r} \right)^{D}, \end{equation}
Here $C_{E} \geq 1$ is a constant independent of $e \in S_{0}$. Of course, by definition of $e \in S_{0}$, the inequality \eqref{form43} holds for with a constant depending on $e$, but we may restrict (and re-normalise) $\sigma$ to a positive measure set to make the constant uniform.

\subsection{Fixing the parameters and refining the tube families from \nref{B1}-\nref{B3}}\label{refining} We now let $\alpha > 0$ and $N \in \N$ be parameters depending on the difference $D - d$ and $\epsilon_{0} > 0$, and we let $Q$ be the collection of dyadic rationals
\begin{displaymath} Q := \{q_{0},q_{1},\ldots,q_{N}\} := \{0,q_{\text{spec}},2^{-N + 2},2^{-N + 3},\ldots,2^{-1},1\}. \end{displaymath}
For concreteness, set
\begin{equation}\label{qSpec} q_{\text{spec}} := 2^{-2N}, \end{equation}
and note that $q_{1} = q_{\text{spec}}$ with this notation; the strange choice of starting with $2^{-N + 2}$ is only needed to achieve $|Q| = N + 1$. We choose $N \geq 1$ so large that
\begin{equation}\label{N} 10N \cdot 2^{-N} < \epsilon_{0}. \end{equation}
As before, we let $\tau > 0$ be a small parameter depending on $\alpha,Q$. The role of $\tau$ will be to mitigate various constants depending on $\alpha$ and $Q$. We now perform the inductive construction from Section \ref{tangentSection}, relative to the measures $\mu$ and $\sigma$, thus finding the objects
\begin{itemize}
\item $\nu = \mu^{B(x_{0},R_{0})} \in \calM(\R^{2})$,
\item $\Delta > 0$ and $0 \leq s \lesssim 1$,
\item $S \subset \spt \sigma \subset S^{1}$ with $\sigma(S) \gtrapprox 1$,
\item $0 \leq s(e) \leq s + \alpha/100$ for $e \in S$. 
\end{itemize}
We also construct the sets $K_{e} \subset \spt \nu \cap B_{0}$, $e \in S$, as in the previous section, recall \eqref{form44}. We present here the properties of $K_{e}$ that we will use (and justify them afterwards):
\begin{itemize}
\item[(K1)\phantomsection\label{K1}] For every $0 \leq k \leq N - 1$, the set $K_{e}$ can be covered by a collection $\calT^{k}(e)$ of dyadic $(\Delta^{q_{N - k}},e)$-tubes such that $\cup \calT^{k - 1}(e) \subset \cup \calT^{k}(e)$ for $1 \leq k \leq N - 1$, and
\begin{displaymath} |\calT^{k}(e)| \approx (\Delta^{q_{N - k}})^{-t_{k}(e)} \quad \text{and} \quad \nu(K_{e} \cap T) \approx (\Delta^{q_{N - k}})^{t_{k}(e)} \text{ for } T \in \calT^{k}(e). \end{displaymath}
Moreover, $|t_{k}(e) - s(e)| \leq \alpha/q_{\text{spec}}$ for $0 \leq k \le N - 1$, and $|\calT^{k - 1}(T^{k})| \approx |\calT^{k - 1}|/|\calT^{k}|$ for $1 \leq k \leq N - 1$,  where $\calT^{k - 1}(T^{k}) = \{T^{k - 1} \in \calT^{k - 1} : T^{k - 1} \subset T^{k}\}$.
\item[(K2)\phantomsection\label{K2}] If $T \in \calT^{0}(e)$ and $x \in \R^{2}$, then $\nu(T) \approx \Delta^{s(e)}$, and
\begin{displaymath} \frac{\nu([K_{e} \cap T] \cap B(x,\Delta^{q_{\text{spec}}}))}{\nu(T)} \lesssim (\Delta^{q_{\text{spec}}})^{d - s(e) - 4\alpha/q_{\text{spec}}}. \end{displaymath}
\end{itemize}
Fix $e \in S$ and, for the moment, write $\tilde{\calT}^{k} := \tilde{\calT}^{k}(e)$, $0 \leq k \leq N - 1$, for the tube collections constructed in the previous section, satisfying \nref{B1}-\nref{B3}. Claim \nref{K2} works for all $T \in \tilde{\calT}^{0}(e) = \calT_{G}$ by Lemma \ref{NCLemma}, noting that $K_{e} \subset K_{G}$ by \eqref{form86} and \nref{B2}.

Some of the claims in \nref{K1} do not work directly for the collections $\tilde{\calT}^{k}(e)$, but they will work for suitable subsets $\calT^{k}(e) \subset \tilde{\calT}^{k}(e)$. The first problem is that nothing in the construction of the collections $\tilde{\calT}^{k}(e)$ guarantees \emph{a priori} that $\cup \tilde{\calT}^{k - 1}(e) \subset \cup \tilde{\calT}^{k}(e)$ for $1 \leq k \leq N - 1$. To obtain this inclusion (claimed in \nref{K1}), we refine the collections $\tilde{\calT}^{k}(e)$ once more "from top down" into the final collections $\calT^{k}(e)$. We remind the reader here that all tubes considered are dyadic, and we omit "$e$" from the notation for the moment. Set $\calT^{N - 1} := \tilde{\calT}^{N - 1}$. Then, let $\calT^{N - 2} := \{T \in \tilde{\calT}^{N - 2} : T \subset \cup \calT^{N - 1}\}$. By \nref{B3}, we infer that
\begin{equation}\label{form95} |\calT^{N - 2}| \approx |\calT^{N - 1}| \cdot \frac{|\tilde{\calT}^{N - 2}|}{|\tilde{\calT}^{N - 1}|} = |\tilde{\calT}^{N - 2}|. \end{equation}
Now, we continue in the same way, including in $\calT^{N - j}$ only those tubes from $\tilde{\calT}^{N - j}$ contained in $\cup \calT^{N - j + 1}$. Repeating the calculation in \eqref{form95}, and assuming inductively that $|\calT^{N - j + 1}| \approx |\tilde{\calT}^{N - j + 1}|$, we find that $|\calT^{N - j}| \approx |\tilde{\calT}^{N - j}|$ for $1 \leq j \leq N - 1$. This completes the construction of the subfamilies $\calT^{k} \subset \tilde{\calT}^{k}$, $0 \leq k \leq N - 1$. It is immediate from the construction that
\begin{equation}\label{form97} \cup \calT^{k} = \bigcap_{j = k}^{N - 1} \cup \tilde{\calT}^{j}, \qquad 0 \leq k \leq N - 1. \end{equation}
It follows from \eqref{form97} and the formula 
\begin{equation}\label{form96} K_{e}^{k} = K_{e}^{k - 1} \cap \left(\cup \tilde{\calT}^{k} \right), \qquad 1 \leq k \leq N - 1, \end{equation}
in \nref{B2} that $K_{e} = K_{e}^{N - 1}$ (as defined in \eqref{form44}) is covered by the tubes in $\calT^{k}$:
\begin{displaymath} K_{e} = K_{e}^{N - 1} = K_{e}^{N - 2} \cap \left( \cup \tilde{\calT}^{N - 1} \right) = \ldots = K_{e}^{0} \cap \bigcap_{j = 1}^{N} \cup \tilde{\calT}^{j} \subset \bigcap_{j = 0}^{N} \tilde{\calT}^{j} \subset \cup \calT^{k}, \quad 0 \leq k \leq N - 1. \end{displaymath}
Now, everything about \nref{K1}-\nref{K2} is clear, except the lower bound for $\nu(K_{e} \cap T)$ in \nref{K1}. In \nref{B2}, we established that $\nu(K_{e}^{k} \cap T) \gtrapprox (\Delta^{q_{N - k}})^{t_{k}(e)}$ for all $T \in \tilde{\calT}^{k}$, but it is generally possible that $K_{e} \subsetneq K_{e}^{k}$ and even $K_{e} \cap T = \emptyset$ for some $T \in \tilde{\calT}^{k}$. This is, in fact, the main reason why we needed to refine $\tilde{\calT}^{k}$ into $\calT^{k}$. Namely, if $T \in \calT^{k}$, we can apply \eqref{form96} repeatedly, and finally \eqref{form97}, to obtain
\begin{align*} K_{e} \cap T & = K_{e}^{N - 1} \cap T = K_{e}^{N - 2} \cap \left(\cup \calT^{N - 1} \right) \cap T\\
& = \ldots = K_{e}^{k} \cap \left( \bigcap_{j = k + 1}^{N - 1} \cup \tilde{\calT}^{j} \right) \cap T = K_{e}^{k} \cap T, \end{align*} 
noting in the last equation that $T \subset \cup \calT^{k}$ is contained in the big intersection by \eqref{form97}. Thus, $\nu(K_{e} \cap T) = \nu(K_{e}^{k} \cap T) \gtrapprox (\Delta^{q_{N - k}})^{t_{k}(e)}$ for $T \in \calT^{k}$, as desired.

\subsubsection{Heuristics: how to contradict the positive dimensionality of $\sigma$?}\label{contradictionHeuristic} We now explain, a little heuristically, how we will contradict the Frostman condition \eqref{frostman} for any $\epsilon_{0} > 0$. A completely rigorous argument is given at the very end of the paper, in Section \ref{conclusion}. Recall that $\sigma(S) \gtrapprox 1$. Hence, there exists an arc $J_{1} \subset S^{1}$ of length 
\begin{equation}\label{form98} \calH^{1}(J_{1}) = \Delta^{q_{2}} = \Delta^{2^{-N + 2}} \quad \text{with} \quad \sigma(J_{1} \cap S) \gtrapprox \Delta^{q_{2}}. \end{equation}
(We use $q_{2}$ here because $q_{1} = q_{\text{spec}}$ will play a somewhat different role than the other elements in $Q$.) After this, we can completely forget about what happens outside $J_{1}$; we aim to show that there is another arc $J_{2} \subset J_{1}$ of length 
\begin{displaymath} \calH^{1}(J_{2}) \sim \Delta^{q_{3}} = \Delta^{-2^{-N + 3}} = (\Delta^{q_{2}})^{2} \end{displaymath}
such that $\sigma(J_{2}) \gtrapprox \sigma(J_{1})$. Then, we will repeat the trick $N - 1$ times to find a single $\Delta^{q_{N}} = \Delta$-arc $J_{N - 1}$ which satisfies $\sigma(J_{N - 1}) \gtrapprox \sigma(J_{N - 2}) \gtrapprox \ldots \gtrapprox \sigma(J_{1})$. It follows from \eqref{form98} that
\begin{displaymath} \Delta^{2^{-N + 2}} = \Delta^{q_{2}} \lessapprox \sigma(J_{1}) \lessapprox \sigma(J_{N - 1}) \lesssim \Delta^{\epsilon_{0}}, \end{displaymath}
using also the Frostman condition \eqref{frostman}. Since $2^{-N + 2}$ is a lot smaller than $\epsilon_{0}$ by the choice made in \eqref{N}, this will give a contradiction.

\subsection{The core argument begins}\label{core} We start by observing that 
\begin{equation}\label{form45} s \leq D + \alpha. \end{equation}
Indeed, this readily follows from the estimate in \nref{K1} for the tube collection $\calT^{0}(e)$, namely
\begin{displaymath} |\calT^{0}(e)| \gtrapprox \Delta^{-t_{0}(e)} = \Delta^{-s(e)}. \end{displaymath}
(Recall from \eqref{form86} that $t_{0}(e) = s(e)$). Since all the tubes in $\calT^{0}$ have positive $\nu$-measure, each of them contains a point in $\spt \nu$. Consequently, recalling that $\nu = \mu^{B(x_{0},R)}$,
\begin{displaymath} \Delta^{-s(e)} \lessapprox N(\pi_{e}(\spt \nu \cap B_{0}),\Delta) = N(\pi_{e}(\spt \mu \cap B(x_{0},R)),R\Delta) \leq C_{E}\left(\frac{R}{R\Delta} \right)^{D} = C_{E}\Delta^{-D}. \end{displaymath}
Since $|s(e) - s| \leq \alpha/100$, we deduce \eqref{form45} if $\Delta > 0$ is sufficiently small. We then pick $0 < \alpha < (d - D)/2$ so \eqref{form45} implies 
\begin{equation}\label{form45b} s < \frac{d + D}{2}. \end{equation}

We now fix any rational 
\begin{displaymath} p := q_{i} \in \{q_{2},\ldots,q_{N - 1}\} \quad  \text{and also write} \quad q := q_{i + 1} = 2p \in Q. \end{displaymath}
We also fix an arc $J_{1} \subset S^{1}$ of length $\ell(J_{1}) = \Delta^{p}$, and another auxiliary parameter 
\begin{equation}\label{epsilonOne} \epsilon_{1} := 2^{-N}. \end{equation}
(If we were short on letters, we could easily replace $\epsilon_{1}$ by $q_{1} = 2^{-N + 2}$ below, but since this $\epsilon_{1}$ has a different role to play than $q_{1}$, we prefer to give it a different letter.) We claim that if $\alpha$ is chosen sufficiently small, depending on $N$ -- which only depends on $\epsilon_{0}$ -- then there exists an arc $J_{2} \subset J_{1}$ of length $\ell(J_{2}) = \Delta^{q} = (\Delta^{p})^{2}$ such that
\begin{equation}\label{form99} \sigma(J_{2} \cap S) \gtrapprox \Delta^{\epsilon_{1}}\sigma(J_{1} \cap S). \end{equation}
We begin the efforts to find $J_{2}$. Since 
\begin{displaymath} \nu(B_{0}) \leq 1, \quad K_{e} \subset B_{0}, \quad \text{and} \quad \nu(K_{e}) \gtrapprox 1 \end{displaymath}
for all $e \in J_{1} \cap S$, we may estimate as follows:
\begin{align*} \int_{J_{1} \cap S} \int_{J_{1} \cap S} \nu(K_{e} \cap K_{e'}) \, d\sigma(e) \, d\sigma(e') & = \int_{B_{0}} \int_{J_{1} \cap S} \int_{J_{1} \cap S} \chi_{K_{e} \cap K_{e'}}(x) \, d\sigma(e) \, d\sigma(e') \, d\nu(x)\\
& = \int_{B_{0}} \left(\int_{J_{1} \cap S} \chi_{K_{e}}(x) \, d\sigma(e) \right)^{2} \, d\nu(x)\\
& \geq \left( \int_{B_{0}} \int_{J_{1} \cap S} \chi_{K_{e}}(x) \, d\sigma(e) \, d\nu(x) \right)^{2}\\
& = \left( \int_{J_{1} \cap S} \nu(K_{e}) \, d\sigma(e) \right)^{2} \gtrapprox \sigma(J_{1} \cap S)^{2}. \end{align*} 
This first implies the existence of $e_{1} \in J_{1} \cap S$ with
\begin{displaymath} \int_{J_{1} \cap S} \nu(K_{e} \cap K_{e_{1}}) \, d\sigma(e) \gtrapprox \sigma(J_{1} \cap S), \end{displaymath}
and then the existence of a subset 
\begin{equation}\label{form100} S_{1} \subset J_{1} \cap S \quad \text{with} \quad \sigma(S_{1}) \gtrapprox \sigma(J_{1} \cap S) \end{equation}
such that
\begin{equation}\label{form46} \nu(K_{e} \cap K_{e_{1}}) \gtrapprox 1, \qquad e \in S_{1}. \end{equation}
Proving the next proposition is the main remaining challenge: it states that a substantial fraction of $\sigma$-mass in $J_{1} \cap S$ is contained surprisingly close to $e_{1}$:
\begin{proposition}\label{prop2} If the parameters 
\begin{displaymath} \alpha = \alpha(C_{E},d,D,N) > 0 \quad \text{and} \quad \tau = \tau(\alpha,C_{E},d,D,N) > 0 \end{displaymath}
are sufficiently small, and $\Delta > 0$ is sufficiently small (that is, $\delta > 0$ was chosen sufficiently small depending on $C_{E},d,D,N$), then
\begin{equation}\label{form47} S_{1} \subset B(e_{1},\Delta^{q - \epsilon_{1}}). \end{equation}
\end{proposition}
Note that \eqref{form100} and \eqref{form47} imply \eqref{form99}, because by \eqref{form47} the set $S_{1} \subset J_{1}$ can be covered by $\lesssim \Delta^{-\epsilon_{1}}$ arcs of length $\Delta^{q}$, contained in $J_{1}$, and one of these arcs, say $J_{2}$, must satisfy 
\begin{displaymath} \sigma(J_{2} \cap S) \geq \sigma(J_{2} \cap S_{1}) \gtrsim \Delta^{\epsilon_{1}}\sigma(S_{1}) \stackrel{\eqref{form100}}{\gtrapprox} \Delta^{\epsilon_{1}}\sigma(J_{1} \cap S). \end{displaymath}
This is \eqref{form99}.

\subsection{Finding a product-like structure inside $K_{e} \cap K_{e_{1}}$}\label{s:productStructure} In proving Proposition \ref{prop2}, we may assume without loss of generality that $e_{1} = (1,0)$, and then we fix $e \in S_{1}$. Recall again the various objects in \nref{K1}-\nref{K2} of Section \ref{refining}. Now we wish to emphasise their dependence on the choice of $e \in S_{1}$, so we write generally write $\calT^{k}(e)$ and $t_{k}(e)$, except for $e = e_{1}$ we continue to write
\begin{displaymath} \calT^{k} := \calT^{k}(e_{1}) \quad \text{and} \quad t_{k} := t_{k}(e_{1}), \quad 0 \leq k \leq N - 1. \end{displaymath} 
Most of the arguments below will take place on the scales $\Delta^{p}$ and $\Delta^{q} = (\Delta^{p})^{2}$, so it will be convenient to have abbreviated notation for tubes of these particular widths. Recall that $\calT^{k}(e)$ is a collection of $\Delta^{q_{N - k}}$-tubes. Let $k_{p} \in \{1,\ldots,N - 2\}$ be the index such that
\begin{equation}\label{form101} q_{N - k_{p}} = p \quad \text{and} \quad q_{N - k_{p} + 1} = q. \end{equation}
We will write
\begin{itemize}
\item $\calT^{th}(e) := \calT^{k_{p}}(e)$ and $t_{th} := t_{k_{p}}(e)$, 
\item $\calT^{nar}(e) := \calT^{k_{p} - 1}(e)$ and $t_{nar} := t_{k_{p} - 1}(e)$,
\end{itemize} 
where "$th$" is short for "thick" and "$nar$" is short for "narrow". As stated, we further omit writing the "$e$" if $e = e_{1}$. 

Recall that all tubes in this paper are subsets of $B_{0} \subset [-1,1)^{2}$, so we can cover $\cup \calT^{th}$ by dyadic subsquares of $[-1,1)^{2}$ of side-length $\Delta^{p}$, which we denote by $\calD_{p}$ in the sequel. Write also
\begin{equation}\label{form50} D_{p} := \Delta^{p} \cdot \Z \cap [-1,1) = \{-1,-1 + \Delta^{p},\ldots,1 - \Delta^{p}\} \end{equation}
for the set of left endpoints of dyadic subintervals of $[-1,1)$ of side-length $\Delta^{p}$. We distinguish some particularly "heavy" squares in $\calD_{p}$. First, write 
\begin{equation}\label{form111} \calG_{p}:= \{R \in \calD_{p} : \nu(R \cap K_{e} \cap K_{e_{1}}) \geq \Delta^{dp + C\tau}\}, \end{equation}
where the constant $C \geq 1$ is determined by the implicit constant in \eqref{form46}, and the implicit constant in the inequality
\begin{displaymath} |\{R \in \calD_{p} : R \cap \spt \nu \neq \emptyset\}| \leq C_{\tau} (\Delta^{p})^{d - \tau}, \end{displaymath}
which follows from \eqref{form39} and Lemma \ref{lemma2}. Then, if $C$ was chosen sufficiently large, 
\begin{displaymath} \sum_{R \in \calD_{p} \setminus \calG_{p}} \nu(R \cap K_{e} \cap K_{e_{1}}) < |\{R \in \calD_{p} : R \cap \spt \nu \neq \emptyset\}| \cdot \Delta^{dp + C\tau} < \frac{\nu(K_{e} \cap K_{e_{1}})}{2}, \end{displaymath}
so at least half of the $\nu$-measure in $K_{e} \cap K_{e_{1}}$ is covered by $G_{p} := \cup \calG_{p}$:
\begin{equation}\label{form54} \nu(G_{p} \cap K_{e} \cap K_{e_{1}}) \gtrapprox 1. \end{equation}
Before proceeding, we perform another refinement of the heavy squares $\calG_{p}$. Namely, we call a square $R \in \calG_{p}$ \emph{bad} if
\begin{displaymath} \nu([R \cap K_{e} \cap K_{e_{1}}] \cap \textbf{Bad}(p,q,e_{1})) \geq \frac{\nu(R \cap K_{e} \cap K_{e_{1}})}{2}. \end{displaymath}
Then
\begin{displaymath} \sum_{R \in \calG_{p} \text{ is bad}} \nu(R \cap K_{e} \cap K_{e_{1}}) \leq 2\nu(\textbf{Bad}(p,q,e_{1})) < 2\Delta^{f(n + 1)\tau} \end{displaymath}
by \eqref{badMeasure} and the disjointness of the squares in $\calG_{p}$. If the function $f$ is sufficiently rapidly increasing, depending on the implicit constant $\sim_{\alpha,Q} 1$ in the exponent of \eqref{form54}, we infer that at most half of the $\nu$-measure of $G_{p} \cap K_{e} \cap K_{e_{1}}$ is covered by the bad squares $R \in \calG_{p}$. Thus, replacing $\calG_{p}$ by the non-bad squares (without changing notation), \eqref{form54} remains true for $G_{p} = \cup \calG_{p}$. Hence, we may assume that
\begin{equation}\label{form55} \nu([R \cap K_{e} \cap K_{e_{1}}] \, \setminus \, \textbf{Bad}(p,q,e_{1})) \gtrapprox \Delta^{dp}, \qquad R \in \calG_{p}. \end{equation}
We claim: it follows from \eqref{form55} that there exists a point $x_{R} \in R \, \setminus \, \textbf{Bad}(p,q,e_{1})$ with the property that
\begin{equation}\label{form56} \nu(B(x_{R},\Delta^{p}) \cap [R \cap K_{e} \cap K_{e_{1}}]) \gtrapprox \Delta^{dp}. \end{equation}
To see this, simply form a $(\Delta^{p}/2)$-net inside the set $(R \cap K_{e} \cap K_{e_{1}}) \, \setminus \, \textbf{Bad}(p,q,e_{1})$. The net evidently just contains $\lesssim 1$ points, since $\ell(R) = \Delta^{p}$. Hence, by \eqref{form55}, one of the net points -- called $x_{R}$ -- must even satisfy
\begin{displaymath} \nu(B(x_{R},\Delta^{p}) \cap [(R \cap K_{e} \cap K_{e_{1}}) \, \setminus \, \textbf{Bad}(p,q,e_{1})]) \gtrapprox \Delta^{dp}, \end{displaymath}
which is a little better than \eqref{form56}.

Recalling that $|\calT^{th}| \lessapprox (\Delta^{p})^{-t_{th}}$ by property \nref{K1}, we now use \eqref{form54} to single out one particularly "heavy" tube $T \in \calT^{th}$. Namely, writing
\begin{displaymath} \calT^{th}_{heavy} := \{T \in \calT^{th} : \nu(T \cap G_{p} \cap K_{e} \cap K_{e_{1}}) \gtrapprox (\Delta^{p})^{t_{th}}\}, \end{displaymath}
and choosing the implicit constant are appropriately (depending on the constants in \nref{K1} and \eqref{form54}), at most half of the $\nu$-mass of $G_{p} \cap K_{e} \cap K_{e_{1}}$ can be covered by the tubes $T \notin \calT^{th}_{heavy}$. Thus, we may find and fix a tube
\begin{equation}\label{form52} T_{0} \in \calT^{th}_{heavy} \quad \text{with} \quad \nu(T_{0} \cap G_{p} \cap K_{e} \cap K_{e_{1}}) \gtrapprox (\Delta^{p})^{t_{th}}. \end{equation}
After this point, the other tubes in $\calT^{th}$ can be completely forgotten. Recalling that $T_{0}$ is a dyadic tube, we note that $T_{0} \cap G_{p}$ is a union of a certain subfamily of $\calG_{p}$, which we denote by $\calG_{T_{0}}$. Since $\nu(R) \lesssim \Delta^{dp}$ for all $R \in \calG_{T_{0}}$ by \eqref{form39} and Lemma \ref{lemma2}, we can infer from \eqref{form52} a lower bound for the cardinality of $\calG_{T_{0}}$:
\begin{displaymath} (\Delta^{p})^{t_{th}} \lessapprox \nu(T_{0} \cap G_{p} \cap K_{e} \cap K_{e_{1}}) \leq \sum_{R \in \calG_{T_{0}}} \nu(R) \lesssim |\calG_{T_{0}}| \cdot \Delta^{dp}, \end{displaymath}
or in other words
\begin{displaymath} |\calG_{T_{0}}| \gtrapprox (\Delta^{p})^{t_{th} - d}. \end{displaymath}
At this point, we extract from $\calG_{T_{0}}$ an arbitrary sub-collection of cardinality $\approx (\Delta^{p})^{t_{th} - d}$, and we keep denoting this collection by $\calG_{T_{0}}$. Thus,
\begin{itemize}
\item[(G1)] $|\calG_{T_{0}}| \approx (\Delta^{p})^{t_{th} - d}$, where $|t_{th} - s(e_{1})| \leq \alpha/q_{\text{spec}}$ by \nref{K1},
\item[(G2)] $R \subset T_{0}$ for all $R \in \calG_{T_{0}}$,
\item[(G3)] $\nu(R \cap K_{e} \cap K_{e_{1}}) \approx \Delta^{dp}$ for all $R \in \calG_{T_{0}}$ since $\calG_{T_{0}} \subset \calG_{p}$.
\end{itemize}
Note here that 
\begin{displaymath} t_{th} \leq s(e_{1}) + \frac{2\alpha}{q_{\text{spec}}} \leq s + \frac{\alpha}{100} + \frac{2\alpha}{q_{1}} \leq D + \frac{3\alpha}{q_{\text{spec}}}. \end{displaymath}
In particular, choosing $\alpha$ small enough, depending on $q_{\text{spec}}$ and $d - D$, we may arrange that $t_{th} - d < 0$, and in particular that $|\calG_{T_{0}}| \geq 1$.

We write $\pi_{1},\pi_{2} \colon \R^{2} \to \R$ for the coordinate projections,
\begin{displaymath} \pi_{1}(x,y) = x \quad \text{and} \quad \pi_{2}(x,y) = y. \end{displaymath}
Recall now the set $D_{p} \subset [-1,1)$ of dyadic rationals from \eqref{form50}, and let $D_{v} \subset D_{p}$ (here "$v$" stands for "vertical") be the left endpoints of the dyadic intervals $\{\pi_{2}(R) : R \in \calG_{T_{0}}\}$, see Figure \ref{fig2}. In fact, it is convenient to introduce the notation $l(I)$ for the left endpoint of an arbitrary (bounded) interval $I \subset \R$, so then we can explicitly write
\begin{equation}\label{form102} D_{v} := \{l(\pi_{2}(R)) : R \in \calG_{T_{0}}\}. \end{equation}
Then
\begin{equation}\label{form51} |D_{v}| = |\calG_{T_{0}}| \approx (\Delta^{p})^{t_{th} - d} \quad \text{with} \quad |t_{th} - s(e_{1})| \leq \frac{\alpha}{q_{\text{spec}}} \end{equation}
by (G1). Next, since $\pi_{1}(T_{0})$ is a dyadic interval of length $\Delta^{p}$, we may apply a rescaling of the form
\begin{equation}\label{affine} A(x,y) = (\Delta^{-p}x,y) + (a_{0},0) \end{equation}
to the effect that $\pi_{1}(A(T_{0})) = [0,1)$, see Figure \ref{fig2}. For notational convenience later on, we assume without loss of generality that $a_{0} = 0$; this corresponds to assuming that $\pi_{1}(T_{0}) = [0,\Delta^{p})$, and yields the simple expression
\begin{equation}\label{form109} \pi_{1}(A(T)) = \Delta^{-p}\pi_{1}(T), \qquad T \in \calT^{nar}(T_{0}). \end{equation}
We now consider the tubes in $\calT^{nar}(T_{0}) = \{T \in \calT^{nar} : T \subset T_{0}\}$. They are dyadic tubes of width $\Delta^{q} = \Delta^{2p}$ with $\pi_{1}$-projection contained in $\pi_{1}(T_{0}) = [0,\Delta^{p})$, so $\{\pi_{1}(A(T)) : T \in \calT^{nar}(T_{0})\}$ is a collection of dyadic subintervals of $[0,1)$ of length $\Delta^{p}$. We write
\begin{equation}\label{form103} D_{h} := \{l(\pi_{1}(A(T))) : T \in \calT^{nar}(T_{0})\} \subset D_{p}. \end{equation}
Here "$h$" is stands for "horizontal". Recall from \nref{K1} and the notational conventions made below \eqref{form101} that
\begin{equation}\label{form61} |D_{h}| = |\calT^{q}(T_{0})| \approx \Delta^{-qt_{nar} + pt_{th}} = \Delta^{-p(2t_{nar} - t_{th})}. \end{equation}
Since further $|t_{th} - s(e_{1})| \leq 2\alpha/q_{\text{spec}}$ and $|t_{nar} - s(e_{1})| \leq 2\alpha/q_{\text{spec}}$ by \nref{K1}, we have the estimate
\begin{equation}\label{form49} |(2t_{nar} - t_{th}) - s(e_{1})| = |2(t_{nar} - s(e_{1})) + (s(e_{1}) - t_{th})| \leq \frac{6\alpha}{q_{\text{spec}}}. \end{equation}
Informally, the combined message from \eqref{form51} and \eqref{form49} is that $D_{v}$ contains roughly $\Delta^{p(s - d)}$ points and $D_{h}$ contains roughly $\Delta^{-ps}$ points, so the product set $D_{h} \times D_{v} \subset [-1,1)^{2}$ contains roughly $\Delta^{-pd}$ points, which are $\Delta^{p}$-separated.
\begin{figure}[h!]
\begin{center}
\includegraphics[scale = 0.4]{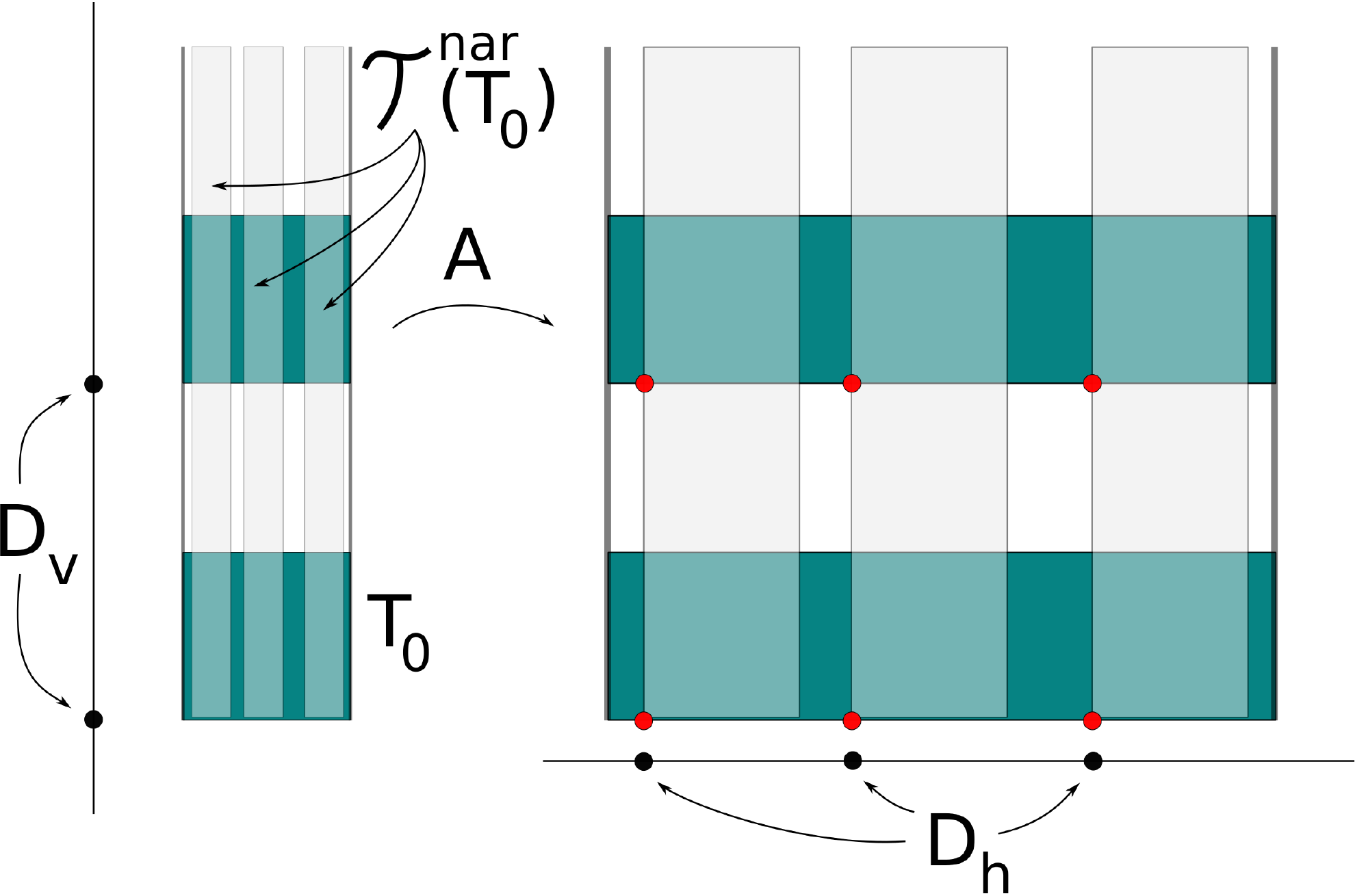}
\caption{The tubes in $\calT^{nar}(T_{0})$ and their images under $A$.}\label{fig2}
\end{center}
\end{figure}

\subsection{Absolute continuity with respect to a product measure}\label{productMeasures} We now consider the following discrete measures:
\begin{equation}\label{muhmuv} \mu_{h} := \frac{1}{|D_{h}|} \sum_{x \in D_{h}} \delta_{x} \quad \text{and} \quad \mu_{v} := \frac{1}{|D_{v}|} \sum_{x \in D_{v}} \delta_{x}. \end{equation} 
We also write $w_{x} := |D_{h}|^{-1}$ and $w^{y} := |D_{v}|^{-1}$,  so that the product measure $\mu_{h} \times \mu_{v}$ on $D_{h} \times D_{v}$ can be written in the form
\begin{displaymath} \mu_{h} \times \mu_{v} = \sum_{(x,y) \in D_{h} \times D_{v}} w_{x}w^{y} \cdot \delta_{(x,y)}. \end{displaymath}
We record that, by \eqref{form51} and \eqref{form61}-\eqref{form49}, we have
\begin{equation}\label{form62} w_{x}w^{y} \approx \Delta^{dp + O_{N}(1)\alpha}. \end{equation} 
Here, and in the sequel, the notation $O_{N}(1)$ refers to a constant with absolute value $\sim_{N} 1$. For example, in the case \eqref{form62} one could explicitly estimate that 
\begin{displaymath} w_{x}w^{y} \stackrel{\eqref{form51} \& \eqref{form61}}{\approx} \Delta^{p(d - t_{th}) + p(2t_{nar} - t_{th})}, \end{displaymath}
where 
\begin{displaymath} \Delta^{dp + 7\alpha p/q_{\text{spec}}} = \Delta^{p[d - s(e_{1})] + ps(e_{1}) + 7\alpha p/q_{\text{spec}}} \leq \Delta^{p(d - t_{th}) + p(2t_{nar} - t_{th})} \leq \Delta^{dp - 7\alpha p/q_{\text{spec}}}, \end{displaymath} 
and $q_{\text{spec}} = 2^{-2N}$. Trying to track the constants in this fashion would soon become exceedingly cumbersome. 

It may appear that the measure $\mu_{h} \times \mu_{v}$ has nothing to do with the "original" measure $\nu$ -- or even its push-forward $A(\nu)$ -- but in fact it does, and this is the next point of investigation. Roughly speaking, we wish to argue that the subset
\begin{displaymath} A(T_{0} \cap G_{p} \cap K_{e} \cap K_{e_{1}}) \end{displaymath} 
has large $(\mu_{h} \times \mu_{v})$-measure, at least after it has been appropriately discretised to $D_{h} \times D_{v}$. Recall that $T_{0} \cap G_{p} \cap K_{e} \cap K_{e_{1}}$ readily has large $\nu$-measure by \eqref{form52}, so we roughly face the problem of showing that $A(\nu) \ll \mu_{h} \times \mu_{v}$ quantitatively. 

We tackle the problem by defining another discrete measure on $D_{h} \times D_{v}$ which \emph{a priori} more faithfully represents $A(\nu)$ than $\mu_{h} \times \mu_{v}$. Consider a point $(x,y) \in D_{h} \times D_{v}$. Then, recalling \eqref{form102} and \eqref{form103}, we have 
\begin{displaymath} x = l(\pi_{1}(A(T))) \quad \text{and} \quad y = l(\pi_{2}(R)) \end{displaymath}
for some $T \in \calT^{nar}(T_{0})$ and some $R \in \calG_{T_{0}}$. We define
\begin{equation}\label{form60} w_{(x,y)} := \frac{\nu([R  \cap K_{e} \cap K_{e_{1}} \cap B(x_{R},\Delta^{p})] \cap T)}{(\Delta^{p})^{t_{th}}} \end{equation}
for these $R = R_{y}$ and $T = T_{x}$, where $x_{R} \in R \, \setminus \, \textbf{Bad}(q_{1},q_{2},e_{1})$ is the point selected at \eqref{form56}. Then, we set
\begin{displaymath} \nu' := \sum_{(x,y) \in D_{h} \times D_{v}} w_{(x,y)} \cdot \delta_{(x,y)}. \end{displaymath}
How close is $\nu'$ to the product measure $\mu_{h} \times \mu_{v}$? The latter gives weight $(|D_{h}||D_{v}|)^{-1} \approx \Delta^{dp + O_{N}(1)\alpha}$ to each pair $(x,y) \in D_{h} \times D_{v}$, so we would like to argue the weights $w_{(x,y)}$ "typically" have the same order of magnitude. This can be accomplished by one more "finding a bad ball" type argument, which we have already seen a few times.

Fix $y \in D_{v}$, and let $R \in \calG_{T_{0}}$ be the square such that $y = l(\pi_{2}(R))$. Then,
\begin{align} \sum_{x \in D_{h}} w_{(x,y)}  & = \sum_{T \in \calT^{nar}(T_{0})} \frac{\nu([R  \cap K_{e} \cap K_{e_{1}} \cap B(x_{R},\Delta^{p})] \cap T)}{(\Delta^{p})^{t_{th}}} \notag\\
&\label{form53} \stackrel{\textup{\nref{K1}}}{=} \frac{\nu(R \cap K_{e} \cap K_{e_{1}} \cap B(x_{R},\Delta^{p}))}{(\Delta^{p})^{t_{th}}} \gtrapprox (\Delta^{p})^{d - t_{th}}, \end{align} 
using first that the tubes $T \in \calT^{nar}(T_{0})$ cover $K_{e_{1}} \cap R \subset K_{e_{1}} \cap T_{0}$ by \nref{K1}, and then recalling \eqref{form56}. Now, using the pigeonhole principle, we find a "typical" value of the weights in $w_{(x,y)}$, $x \in D_{h}$. In other words, first inferring the trivial upper bound
\begin{displaymath} w_{(x,y)} \leq \frac{\nu(R)}{(\Delta^{p})^{t_{th}}} \leq (\Delta^{p})^{d - t_{th}}, \end{displaymath}
from the Frostman condition for $\nu$, we find $\eta \in [0,d - t_{th}]$ and a further subset $D^{y}_{h} \subset D_{h}$ with the properties that
\begin{equation}\label{form59} \sum_{x \in D_{h}^{y}} w_{(x,y)} \approx \sum_{x \in D_{h}} w_{(x,y)} \gtrapprox (\Delta^{p})^{d - t_{th}} \quad \text{and} \quad \Delta^{p\eta} \leq w_{(x,y)} \leq 2\Delta^{p\eta} \text{ for all } x  \in D_{h}^{y}. \end{equation}
Here the number $\Delta^{p\eta}$ should be interpreted as the "typical value" of the constants $w_{(x,y)}$, $x \in D_{h}$, written as a power of $\Delta$ for clarity. Next, using once more the Frostman estimate for $\nu$, we infer that
\begin{displaymath} |D_{h}^{y}| \cdot \Delta^{p\eta} \leq \sum_{x \in D_{h}} w_{(x,y)} \leq \sum_{T \in \calT^{nar}(T_{0})} \frac{\nu(R \cap T)}{(\Delta^{p})^{t_{th}}}  \leq (\Delta^{p})^{d - t_{th}}, \end{displaymath}
whence 
\begin{equation}\label{form57} |D_{h}^{y}| \lesssim (\Delta^{p})^{d - t_{th} - \eta}. \end{equation}
Now, we claim that 
\begin{equation}\label{form58} \eta \geq d - \frac{4\alpha}{q_{\text{spec}}}. \end{equation}  
Assume to the contrary: $\eta < d - 4\alpha/q_{\text{spec}}$. Recalling from \eqref{form51} (or \nref{K1}) that $|t_{th} - s(e_{1})| \leq 2\alpha/q_{\text{spec}}$, and that $|s(e_{1}) - s| < \alpha$, we see that
\begin{equation}\label{form104} \eta + t_{th} - d < s(e_{1}) - \frac{2\alpha}{q_{\text{spec}}} < s - 2\alpha.  \end{equation} 
Now, let $\calT_{y}$ be the collection of tubes in $\calT^{nar}(T_{0})$ corresdponding to the points in $D_{h}^{y}$. More precisely, recall that every $x \in D_{h}^{y} \subset D_{h}$ has the form $x = l(\pi_{1}(A(T)))$ for some $T \in \calT^{nar}(T_{0})$, and we denote the tubes of $\calT^{nar}(T_{0})$ so obtained by $\calT_{y}$. With this notation,
\begin{align} \sum_{T \in \calT_{y}} \nu(T \cap B(x_{R},\Delta^{p})) & \geq (\Delta^{p})^{t_{th}}\sum_{T \in \calT_{y}} \frac{\nu([R \cap K_{e} \cap K_{e_{1}} \cap B(x_{R},\Delta^{p})] \cap T)}{(\Delta^{p})^{t_{th}}} \notag\\
&\label{form120} = (\Delta^{p})^{t_{th}} \sum_{x \in D_{h}^{y}} w_{(x,y)} \stackrel{\eqref{form59}}{\approx} (\Delta^{p})^{t_{th}} \sum_{x \in D_{h}} w_{(x,y)} \gtrapprox \Delta^{dp},  \end{align}
recalling \eqref{form53} in the last estimate. In other words, the collection $\calT_{y}$ of $(\Delta^{q},e_{1})$-tubes of cardinality
\begin{displaymath} |\calT_{y}| = |D_{h}^{y}| \stackrel{\eqref{form57} \& \eqref{form104}}{\lesssim} (\Delta^{p})^{-s + 2\alpha} \stackrel{q = 2p}{=} (\Delta^{q - p})^{-s + 2\alpha} \end{displaymath}
covers a set of $\nu$-measure $\gtrapprox \Delta^{pd}$ inside the ball $B(x_{R},\Delta^{p}) =: B$. Consequently, a set of $\nu^{B}$-measure $\gtrapprox 1$ can be covered by a family $\calT_{n + 1}(e_{1})$ of $(\Delta^{q - p},e_{1})$-tubes of cardinality $|\calT_{n + 1}(e_{1})| \lesssim (\Delta^{q - p})^{-s + 2\alpha}$. Now we may repeat an argument we have already seen many times (for example right after \eqref{form38}): assuming that $f$ is rapidly increasing enough -- depending on the implicit constants in the lower bound on line \eqref{form120} -- and using Lemma \ref{pigeon}, we infer that $B$ is an $e_{1}$-bad ball relative to the scale $\Delta^{q}$. In particular $x_{R} \in \textbf{Bad}(p,q,e_{1})$, contrary to the choice of $x_{R}$ above \eqref{form56}. This contradiction establishes \eqref{form58}.

The upshot is that for any $y \in D_{v}$ fixed, there exists a subset $D_{h}^{y} \subset D_{h}$ such that \eqref{form59} holds, and
\begin{displaymath} w_{(x,y)} \sim \Delta^{p\eta} \leq (\Delta^{p})^{d - 4\alpha/q_{\text{spec}}} \lessapprox \Delta^{-C_{N}\alpha}w_{x}w^{y}, \end{displaymath} 
recalling \eqref{form62} in the final estimate; here $C_{N} \geq 1$ is a constant depending only on $N$. We write
\begin{displaymath} G := \bigcup_{y \in D_{v}} D_{h}^{y}, \end{displaymath}
and then consider the restriction of $\nu'$ to $G$:
\begin{displaymath} \nu_{G} := \sum_{(x,y) \in G} w_{(x,y)} \cdot \delta_{(x,y)}. \end{displaymath} 
Then $\nu_{G} \ll \mu_{h} \times \mu_{v}$ with density 
\begin{equation}\label{form63} \frac{d\nu_{G}}{d(\mu_{h} \times \mu_{v})} \lessapprox \Delta^{-C_{N}\alpha}, \end{equation}
and
\begin{equation}\label{form64} \nu_{G}(D_{h} \times D_{v}) = \sum_{y \in D_{v}} \sum_{x \in D_{h}^{y}} w_{(x,y)} \stackrel{\eqref{form59}}{\gtrapprox} |D_{v}| \cdot (\Delta^{p})^{d - t_{th}} \stackrel{\eqref{form51}}{\gtrapprox} 1. \end{equation}
From \eqref{form63}-\eqref{form64}, we finally infer that
\begin{equation}\label{form67} (\mu_{h} \times \mu_{v})(\spt \nu_{G}) \gtrapprox \Delta^{C_{N}\alpha}. \end{equation}

\subsection{Projecting the measures $\mu_{h} \times \mu_{v}$ and $\nu_{G}$} Recall that we are in the process of proving Proposition \ref{prop2}: for a fixed vector $e \in S_{1} \subset J_{1}$, we are trying to show that 
\begin{equation}\label{form68} |e - (1,0)| = |e - e_{1}| \leq \Delta^{q - \epsilon_{1}}, \end{equation}
where $\epsilon_{1} = 2^{-N}$ according to the choice made in \eqref{epsilonOne}. This will be true if $\alpha = \alpha(N) > 0$ and $\tau = \tau(\alpha,N) > 0$ are chosen sufficiently small. We now make a counter assumption:
\begin{equation}\label{form106} \Delta^{q - \epsilon_{1}} < |e - e_{1}| \lesssim \Delta^{p}, \end{equation}
where the upper bound follows from $e \in J_{1}$.

So far, the role of the vector $e$ has been passive, but now we concentrate on it. Recall from \nref{K1} that the set $K_{e}$ is contained in the union of the $(\Delta^{q},e)$-tubes in the collection $\calT^{nar}(e)$. We want to say something a little sharper concerning the intersection $K_{e} \cap T_{0}$: because $|e - e_{1}| \leq \diam(S_{1}) \lesssim \Delta^{p}$, and $T_{0} \in \calT^{th} = \calT^{th}(e_{1})$ is a tube of width $\Delta^{p}$, we first note that 
\begin{displaymath} |\{T \in \calT^{th}(e) : T \cap T_{0} \neq \emptyset\}| \lesssim 1. \end{displaymath}
The tube $T_{0}$ and one of the $\lesssim 1$ tubes in $T \in \calT^{th}(e)$ with $T \cap T_{0} \neq \emptyset$ are shown in Figure \ref{fig1}. So, $K_{e} \cap T_{0}$ is covered by the union of the tubes in $\calT^{nar}(e)$ contained in one of $\lesssim 1$ tubes in $\calT^{th}(e)$. We denote this collection by $\calT^{nar}(e,T_{0})$. Recalling \nref{K1}, and that $q = 2p$, we then infer that
\begin{equation}\label{form66} |\calT^{nar}(e,T_{0})| \lessapprox \Delta^{-t_{nar}(e)q + t_{th}(e)p} = \Delta^{-p(2t_{nar}(e) - t_{th}(e))} = (\Delta^{p})^{-s + O_{N}(1)\alpha}, \end{equation}
since (repeating \eqref{form49}), we have
\begin{displaymath} |(2t_{nar}(e) - t_{th}(e)) - s(e)| \leq \frac{6\alpha}{q_{\text{spec}}}. \end{displaymath}
\begin{figure}[h!]
\begin{center}
\includegraphics[scale = 0.4]{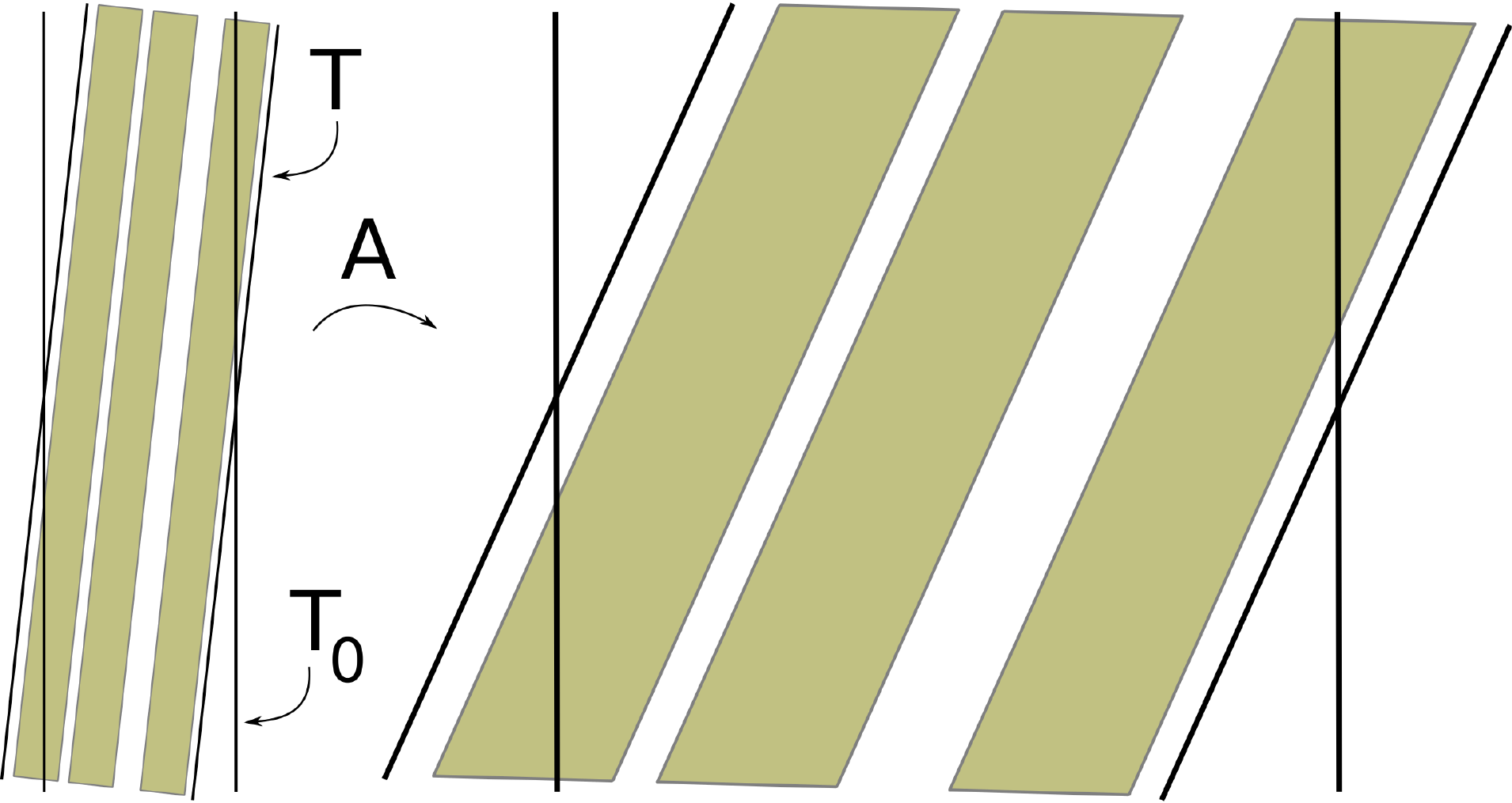}
\caption{Some tubes in $\calT^{nar}(e,T_{0})$ and their images under the map $A$.}\label{fig1}
\end{center}
\end{figure}

We next claim that
\begin{equation}\label{form65} \spt \nu_{G} = \{(x,y) \in D_{h} \times D_{v} : w_{(x,y)} > 0\} \subset \bigcup_{T \in \calT^{nar}(e,T_{0})} A(T)(C\Delta^{p}), \end{equation}
where $C \geq 1$ is an absolute constant, and $A(T)(C\Delta^{p})$ means the $C\Delta^{p}$-neighbourhood of $A(T)$. The sets $A(T)(C\Delta^{p})$ are not exactly tubes in the strict sense of this paper, but they are each contained in an \emph{ordinary $(C\Delta^{p},e')$-tube}, where
\begin{equation}\label{form69} |e_{1} - e'| \sim \Delta^{-p}|e_{1} - e|. \end{equation} 
By an ordinary $(w,e')$-tube, we mean a set of the form $\pi_{e'}^{-1}(I)$, where $I \subset \R$ and $\ell(I) = w$. For a proof of these claims on the geometry of $A(T)$, see Figure \ref{fig3}.
\begin{figure}[h!]
\begin{center}
\includegraphics[scale = 0.4]{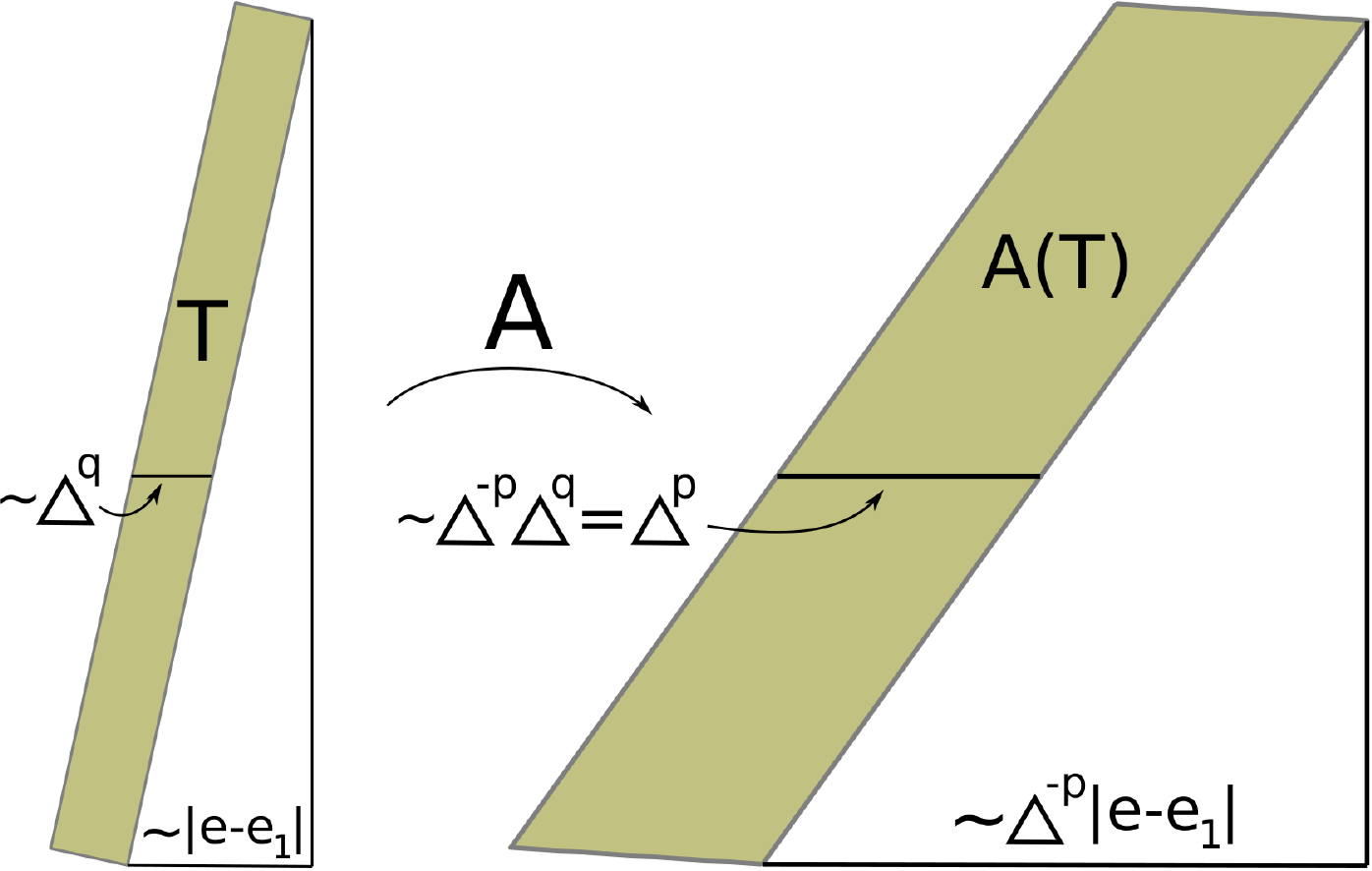}
\caption{The geometry of the images $A(T)$ for $T \in \calT^{nar}(e,T_{0})$.}\label{fig3}
\end{center}
\end{figure}

%We discuss this claim about the geometry of $A(T)(C\Delta^{p})$ before proving \eqref{form65}. From \eqref{form106}, we first infer that
%\begin{equation}\label{form105} \Delta^{q - \epsilon_{1}} \lesssim |e^{2}| \lesssim \Delta^{p}, \quad \text{where} \quad (e^{1},e^{2}) := e. \end{equation}
%Fix $T = \pi_{e}^{-1}(I) \cap B_{0} \in \calT^{nar}(e,T_{0})$, where $I \subset \R$ is a dyadic interval with $\ell(I) = \Delta^{q}$, and define an affine map $A^{\ast}(x,y) = (x,\Delta^{-p}y)$. Then, it is easy to check that
%\begin{displaymath} A(T) \subset A(\pi_{e}^{-1}(I)) = \pi_{e'}^{-1}(I'), \end{displaymath} 
%where $e' = A^{\ast}(e)/|A^{\ast}(e)|$, and $I' \subset \R$ is an interval of length $\ell(I') \sim \Delta^{-p}\ell(I) = \Delta^{p}$. It follows that $A(T)(C\Delta^{p})$ is contained in the ordinary $(C\Delta^{p},e')$-tube $\pi_{e'}^{-1}(I')$ and \eqref{form69} is also easy to verify from the explicit formula for $e'$, and \eqref{form105}.

In combination with \eqref{form66}, \eqref{form65} will therefore have the following corollary:
\begin{proposition}\label{prop1} $N(\pi_{e'}(\spt \nu_{G}),\Delta^{p}) \lessapprox (\Delta^{p})^{-s - C_{N}\alpha}$, where $e' \in S^{1}$ satisfies \eqref{form69}. \end{proposition}

We then prove \eqref{form65}. Pick $(x,y) \in D_{h} \times D_{v}$ with $w_{(x,y)} > 0$. Let $R \in \calG_{T_{0}}$ and $T_{e_{1}} \in \calT^{nar}(T_{0})$ be such that
\begin{displaymath} y = l(\pi_{2}(R)) \quad \text{and} \quad x = l(\pi_{2}(A(T_{e_{1}}))). \end{displaymath}
Then $(x_{0},y_{0}) := A^{-1}(x,y) = (\Delta^{p}x,y) \in R \cap T_{e_{1}}$. Moreover, recalling the definition \eqref{form60},
\begin{displaymath} w_{(x,y)} > 0 \quad \Longrightarrow \quad \nu(R \cap K_{e} \cap T_{e_{1}}) > 0, \end{displaymath} 
so in particular there exists a tube $T_{e} \in \calT^{nar}(e,T_{0})$ such that $\nu(R \cap T_{e_{1}} \cap T_{e}) > 0$. Now, pick a point $(x',y') \in R \cap T_{e_{1}} \cap T_{e}$, and note that  
\begin{displaymath} |x_{0} - x'| \leq \Delta^{q} \quad \text{and} \quad |y_{0} - y'| \leq \Delta^{p}. \end{displaymath}
It follows that 
\begin{displaymath} |(x,y) - A(x',y')| = |A(x_{0},y_{0}) - A(x',y')| = |(\Delta^{-p}(x_{0} - x'),y_{0} - y')| \lesssim \Delta^{p}. \end{displaymath}
Because $A(x',y') \in A(T_{e})$, we infer that also $(x,y) \in A(T_{e})(C\Delta^{p})$ for some absolute constant $C \geq 1$. This proves \eqref{form65}. 

We next aim to use Proposition \ref{prop1} to derive a contradiction from the lower bound in \eqref{form106}. First, from \eqref{form69} and \eqref{form106}, we infer that
\begin{displaymath} |e' - e_{1}| \sim \Delta^{-p}|e - e_{1}| \geq \Delta^{q - p - \epsilon_{1}} = \Delta^{p - \epsilon_{1}} \quad \text{and} \quad |e' - e_{1}| \lesssim 1. \end{displaymath}
Now, fix $\theta \in \R$ such that $(1,\theta) \| e'$. Then, $ \Delta^{p - \epsilon_{1}} \lesssim |\theta| \lesssim 1$, and for slight notational convenience, we work under the assumption that
\begin{equation}\label{form110} \theta \gtrsim \Delta^{p - \epsilon_{1}}. \end{equation}
We now wish to compute some $L^{2}$-norms of the measures $\mu_{h}$ and $\mu_{h} \ast \theta \mu_{v}$, where $\theta \mu_{v}$ refers to the push-forward of $\mu_{v}$ under the map $x \mapsto \theta x$. These are discrete measures, so their $L^{2}$-norm, literally speaking, is infinity. However, we can obtain useful information by mollifying the measures first at scale $\Delta^{p}$. To this end, let $\psi := \tfrac{1}{2}\chi_{[-1,1]}$, and for $\rho > 0$, define $\psi_{\rho}(x) := \rho^{-1} \psi(x/\rho) = \tfrac{1}{2\rho} \chi_{[-\rho,\rho]}$. Then, recalling that $\mu_{h}$ was a (normalised) sum of Dirac measures supported on the $\Delta^{p}$-separated set $D_{h}$, see \eqref{muhmuv}, it is easy to see that
\begin{equation}\label{form108} \|\mu_{h} \ast \psi_{\Delta^{p}}\|_{L^{2}} \sim \left(\frac{1}{\Delta^{p}|D_{h}|}\right)^{1/2} \approx (\Delta^{p})^{(s - 1)/2 + O_{N}(1)\alpha}, \end{equation}
recalling from \eqref{form61}-\eqref{form49} that $|D_{h}| \approx (\Delta^{p})^{-s + O_{N}(1)\alpha}$ in the last estimate. Next, we investigate the $L^{2}$-norm of the convolution $\mu_{h} \ast \theta \mu_{v}$. From the choice of $\theta$, namely $(1,\theta) \| e'$, one can easily verify that
\begin{displaymath} \mu_{h} \ast \theta \mu_{v} = |(1,\theta)|_{\sharp}[\pi_{e'}(\mu_{h} \times \mu_{v})], \end{displaymath}  
so (using also $|(1,\theta)| \sim 1$), we infer that
\begin{equation}\label{form107} \|(\mu_{h} \ast \theta \mu_{v}) \ast \psi_{\Delta^{p}}\|_{L^{2}} \sim \|\pi_{e'}(\mu_{h} \times \mu_{v}) \ast \psi_{\Delta^{p}}\|_{L^{2}}.  \end{equation}
To estimate the quantity on the right hand side, we start by noting that the support of the measure
\begin{displaymath} \pi_{e'}[(\mu_{h} \times \mu_{v})|_{\spt \nu_{G}}] \ast \psi_{\Delta^{p}} \end{displaymath}
is contained in the $2\Delta^{p}$-neighbourhood of the set $\pi_{e'}(\spt \nu_{G})$, and hence, by Proposition \ref{prop1}, has Lebesgue measure no larger than $\lessapprox (\Delta^{p})^{1 - s - C_{N}\alpha}$. Consequently, using \eqref{form67} (plus the fact that neither push-forward nor convolution with $\psi_{\Delta^{p}}$ affects total variation), and then the Cauchy-Schwarz inequality, we obtain
\begin{displaymath} \Delta^{C_{N}\alpha} \stackrel{\eqref{form67}}{\lessapprox} \|\pi_{e'}[(\mu_{h} \times \mu_{v})|_{\spt \nu_{G}}] \ast \psi_{\Delta^{p}}\|_{L^{1}} \lessapprox \left( (\Delta^{p})^{1 - s - C_{N}\alpha} \right)^{1/2}\|\pi_{e'}(\mu_{h} \times \mu_{v}) \ast \psi_{\Delta^{p}}\|_{L^{2}}.  \end{displaymath} 
Combining this estimate with \eqref{form108}-\eqref{form107}, we have now established that
\begin{equation}\label{form70} \|(\mu_{h} \ast \theta\mu_{v}) \ast \psi_{\Delta^{p}}\|_{L^{2}} \gtrapprox \Delta^{C_{N}\alpha} \|\mu_{h} \ast \psi_{\Delta^{p}}\|_{L^{2}}. \end{equation} 
The estimate \eqref{form70} will soon place us in a position to apply Shmerkin's inverse theorem, \cite[Theorem 2.1]{Sh}, the relevant parts of which are also stated as Theorem \ref{shmerkin} below. Before doing so, we make some remarks. First, note from \eqref{form109}-\eqref{form103} that
\begin{displaymath} \Delta^{p} \cdot [\spt \mu_{h}] = \Delta^{p} \cdot D_{h} = \{l(\pi_{1}(T))) : T \in \calT^{nar}(T_{0})\}. \end{displaymath}
In particular, since $\nu(K_{e_{1}} \cap T) > 0$ for all $T \in \calT^{nar}(T_{0})$ by \nref{K1}, we have
\begin{equation}\label{form71} \Delta^{p} \cdot [\spt \mu_{h}] \subset \pi_{1}(T_{0}) \cap [\pi_{1}(K_{e_{1}})](\Delta^{q}). \end{equation}
Now, we apply the facts that $e_{1} \in S_{1} \subset \spt \sigma$ and $\nu = \mu^{B(x_{0},R_{0})}$, which imply that \eqref{form43} holds for $e_{1}$, and for $\spt \nu$ in place of $K$:
\begin{displaymath} N(\pi_{e_{1}}(\spt \nu \cap B_{0}) \cap B(x,R),r) \leq C_{E} \left(\frac{R}{r} \right)^{D} \end{displaymath}
for all $x \in \R$ and $0 < r < R < \infty$. In particular, the estimate above holds for all $B(x,R) \subset \pi_{1}(T_{0})$ and all $\Delta^{q} < r < R \leq \Delta^{p}$. It follows from this, $K_{e_{1}} \subset \spt \nu \cap B_{0}$, and \eqref{form71} that 
\begin{equation}\label{form72} N([\spt \mu_{h}] \cap B(x,R),r) \lesssim C_{E} \left(\frac{R}{r} \right)^{D}, \qquad x \in \R, \: \Delta^{p} \leq r < R \leq 1. \end{equation}
Here $0 < D < \min\{d,1\}$, so \eqref{form72} means that the support of $\mu_{h}$ is porous on all scales between $\Delta^{p}$ and $1$. This is good news in view of applying Shmerkin's inverse theorem, but we also need to know something about the measure $\theta \mu_{v}$, namely that it cannot be concentrated on a very small number of $\Delta^{p}$-intervals.

\subsection{Non-concentration of $\theta \mu_{v}$} The goal in this section is to show that
\begin{equation}\label{form78} (\theta \mu_{v})(I) \lessapprox \Delta^{q_{\text{spec}}(d - s) - C_{N}\alpha} \end{equation}
for any interval $I \subset \R$ of length $\ell(I) = \Delta^{p}$. Here we need to know that 
\begin{displaymath} q_{\text{spec}} = 2^{-2N} < 2^{-N} = \epsilon_{1}, \end{displaymath}
recall the choices \eqref{qSpec} and \eqref{epsilonOne}. Then, recalling from \eqref{form45b} that $s < (d + D)/2$, taking $\alpha, \Delta > 0$ sufficiently small in terms of $N$, we will find that
\begin{equation}\label{form113} (\theta \mu_{v})(I) \leq \Delta^{q_{\text{spec}}(d - D)/3} \qquad \text{ for all } I \subset \R \text{ with } \ell(I) = \Delta^{p}. \end{equation} 
Recall from \eqref{form110} that $\theta \gtrsim \Delta^{p - \epsilon_{1}}$, so \eqref{form78} will follow once we manage to prove that
\begin{equation}\label{form73} \mu_{v}(I) \lessapprox \Delta^{q_{\text{spec}}(d - s(e_{1})) - 4\alpha/q_{\text{spec}}} \end{equation}
for all intervals $I \subset \R$ of length $\Delta^{\epsilon_{1}}$. Furthermore, since $q_{\text{spec}} < \epsilon_{1}$, it suffices to verify \eqref{form73} for all dyadic intervals of length $\Delta^{q_{\text{spec}}}$. We fix one such interval $I$. Recall from \eqref{muhmuv} the definition of $\mu_{v}$:
\begin{displaymath} \mu_{v} = \frac{1}{|D_{v}|} \sum_{x \in D_{v}} \delta_{x} \stackrel{\eqref{form51}}{\approx} (\Delta^{p})^{d - t_{th}} \sum_{x \in D_{v}} \delta_{x}. \end{displaymath}
For each $x \in D_{v} \cap I$, let $R \in \calG_{T_{0}}$ be the heavy square such that $R \subset T_{0}$ and $x = l(\pi_{2}(R))$; then, since $I$ is a dyadic interval, we have $\pi_{2}(R) \subset I$, and hence $R \subset \pi_{2}^{-1}(I)$. It follows that 
\begin{displaymath} |D_{v} \cap I| \leq \card \{R \in \calG_{T_{0}} : R \subset \pi_{2}^{-1}(I)\}. \end{displaymath} 
Next, by the definition of heavy squares in \eqref{form111}, we recall that
\begin{displaymath} \nu(R \cap K_{e_{1}}) \gtrapprox \Delta^{dp}, \end{displaymath}
and consequently
\begin{equation}\label{form75} \mu_{v}(I) \approx (\Delta^{p})^{d - t_{th}} \cdot |D_{v} \cap I| \lessapprox (\Delta^{p})^{-t_{th}} \nu(T_{0} \cap K_{e_{1}} \cap \pi_{2}^{-1}(I)). \end{equation}
We recall from \nref{K1} that the set $T_{0} \cap K_{e_{1}}$ is covered by the $(\Delta,e_{1})$-tubes in $\calT^{0}(T_{0}) = \{T \in \calT^{0} : T \subset T_{0}\}$, and consequently
\begin{equation}\label{form76} \nu(T_{0} \cap K_{e_{1}} \cap \pi_{2}^{-1}(I)) = \sum_{T \in \calT^{0}(T_{0})} \nu(T \cap K_{e_{1}} \cap \pi_{2}^{-1}(I)). \end{equation}
Further, by iterating the branching estimate $|\calT^{k - 1}(T^{k})| \approx |\calT^{k - 1}|/|\calT^{k}|$ in \nref{K1} in the same manner as we did in \eqref{form34}, we have
\begin{equation}\label{form77} |\calT^{0}(T_{0})| \approx \frac{|\calT^{0}|}{|\calT^{th}|} \approx \Delta^{-s(e_{1}) + pt_{th}},  \end{equation} 
recalling also the choice of the number $t_{th}$ from under \eqref{form101}. We now fix a tube $T \in \calT^{0}(T_{0})$, and note that the intersection $T \cap \pi_{2}^{-1}(I)$ can be covered by $\lesssim 1$ balls of radius $\Delta^{q_{\text{spec}}}$. Now, we finally use the non-concentration estimate from \nref{K2}, which we repeat here for convenience:
\begin{displaymath} \frac{\nu([K_{e_{1}} \cap T] \cap B(x,\Delta^{q_{\text{spec}}}))}{\nu(T)} \lesssim (\Delta^{q_{\text{spec}}})^{d - s(e) - 4\alpha/q_{\text{spec}}}, \qquad x \in \R^{2}, \: T \in \calT^{0}. \end{displaymath}
Recalling (this is also stated in \nref{K2}) that $\nu(T) \approx \Delta^{s(e_{1})}$ for $T \in \calT^{0}$, we may combine the estimate above with \eqref{form75}-\eqref{form77} to obtain
\begin{align*} \mu_{v}(I) \lessapprox (\Delta^{p})^{-t_{th}} \cdot |\calT^{0}(T_{0})| \cdot \Delta^{s(e_{1}) + q_{\text{spec}}(d - s(e_{1})) - 4\alpha/q_{\text{spec}}} \lessapprox \Delta^{q_{\text{spec}}(d - s(e_{1})) - 4\alpha/q_{\text{spec}}}, \end{align*}
which is precisely \eqref{form73}.

\subsection{Applying Shmerkin's inverse theorem} Now, we have gathered all the pieces to apply Shmerkin's inverse thereorem \cite[Theorem 2.1]{Sh}, whose statement (in reduced form) we also include right here for the reader's convenience. We explain the notions appearing in the theorem afterwards.
\begin{thm}[Shmerkin]\label{shmerkin} Given $\beta > 0$ and $m_{0} \in \N$, there are $\kappa > 0$ and $m \geq m_{0}$ such that the following holds for all large enough $\ell$. Let $\Delta = 2^{-\ell m}$, and let $\mu,\nu$ be $\Delta$-measures such that
\begin{equation}\label{inverseHyp} \|\mu \ast \nu\|_{L^{2},\textup{Sh}} \geq \Delta^{\kappa}\|\mu\|_{L^{2},\textup{Sh}}. \end{equation}
Then, there exist sets $A \subset \spt \mu$ and $B \subset \spt \nu$ such that
\begin{itemize}
\item[\textup{(A)}] there is a sequence $(R_{s}^{\mu})_{s = 0}^{\ell} \subset \{1,\ldots,2^{m}\}^{\ell + 1}$, such that 
\begin{displaymath} N(A \cap I,2^{-(s + 1)m}) = R_{s}^{\mu} \end{displaymath}
for all dyadic intervals $I$ of length $2^{-ms}$ intersecting $A$,
\item[\textup{(B)}] there is a sequence $(R_{s}^{\nu})_{s = 0}^{\ell} \subset \{1,\ldots,2^{m}\}^{\ell + 1}$, such that 
\begin{displaymath} N(B \cap I,2^{-(s + 1)m}) = R_{s}^{\nu} \end{displaymath}
for all dyadic intervals $I$ of length $2^{-ms}$ intersecting $B$.
\end{itemize}
For each $s \in \{0,\ldots,\ell\}$, either $R_{s}^{\nu} = 1$ or $R_{s}^{\mu} \geq 2^{(1 - \beta)m}$, and the set $\calS = \{s : R_{s}^{\mu} \geq 2^{(1 - \beta)m}\}$ satisfies
\begin{equation}\label{form116} m|\calS| \geq \log \|\nu\|_{L^{2},\textup{Sh}}^{-2} + \beta \log_{2} \Delta. \end{equation}
\end{thm}
Now, we explain the concepts appearing above. First, for $\Delta \in 2^{-\N}$, a \emph{$\Delta$}-measure is any probability measure in $\calM(\Delta \cdot \Z \cap [-1,1))$. In our case, we will actually be concerned with $\Delta^{p}$-measures, such as $\mu_{h}$. For a $\Delta$-measure $\mu \in \calM(\Delta \cdot \Z \cap [-1,1))$, Shmerkin defines the (non-standard) $L^{2}$-norm
\begin{displaymath} \|\mu\|_{L^{2},\textup{Sh}}^{2} := \sum_{x \in \Delta \cdot \Z \cap [-1,1)} \mu(\{x\})^{2}. \end{displaymath} 
It is easy to see that
\begin{equation}\label{form119} \|\mu\|_{L^{2},\textup{Sh}}^{2} \sim \Delta \cdot \|\mu \ast \psi_{\Delta}\|_{L^{2}}^{2}, \qquad \mu \in \calM(\Delta \cdot \Z \cap [-1,1)). \end{equation} 
Of the measures we are interested in presently, $\mu_{h}$ is already a $\Delta^{p}$-measure, but $\theta \mu_{v}$ is not. However, we can associate to $\theta \mu_{v}$ a $\Delta^{p}$-measure in the following canonical way:
\begin{displaymath} (\theta \mu_{v})' := \sum_{x \in D_{p}} (\theta \mu_{v})([x,x + \Delta^{p})) \cdot \delta_{x}.  \end{displaymath} 
Then, it follows from \eqref{form113} that $(\theta \mu_{v})'(\{x\}) \leq \Delta^{\eta p}$ for $\eta = q_{\textup{spec}}(d - D)/3$, and consequently (noting that $(\theta \mu_{v})'$ is a probability measure)
\begin{displaymath} \|(\theta \mu_{v})'\|_{L^{2},\textup{Sh}}^{2} = \sum_{2^{-j} \leq \Delta^{\eta p}} \mathop{\sum_{x \in D_{p}}}_{2^{-j - 1} < (\theta \mu_{v})'(\{x\}) \leq 2^{-j}} (\theta \mu_{v})'(\{x\})^{2} \leq 2\Delta^{\eta p}. \end{displaymath}
As a technical corollary, noting also that $\Delta^{-p} \geq |D_{p}|/2$, we record that
\begin{equation}\label{form115} \log \|(\theta \mu_{v})'\|_{L^{2},\textup{Sh}}^{-2} \geq \eta \log \Delta^{-p} - 2 = \eta \log |D_{p}| - 3. \end{equation} 
We further record the following consequence of \eqref{form70} and \eqref{form119}:
\begin{equation}\label{form114} \|\mu_{h} \ast (\theta \mu_{v})'\|_{L^{2},\textup{Sh}} \gtrapprox \Delta^{C_{N}\alpha}\|\mu_{h}\|_{L^{2},\textup{Sh}}, \end{equation}  
Then, we apply Shmerkin's inverse theorem to the measures $\mu_{h}$ and $(\theta \mu_{v})'$, for any 
\begin{displaymath} 0 < \beta < \min\{\eta/2,(1 - D)\} = \min\{q_{\text{spec}}(d - D)/6,(1 - D)\}, \end{displaymath}
and for some large $m_{0} \in \N$ to be prescribed in a moment, depending only on $1 - D$ and the constant $C_{E}$ in \eqref{form72}. The inverse theorem then produces the constants
\begin{displaymath} m = m(\beta,m_{0}) \geq m_{0} \quad \text{and} \quad \kappa = \kappa(\beta,m_{0}) > 0. \end{displaymath}
Note that the choice of $\beta$ can be made depending only on $q_{\text{spec}} = 2^{-2N},d - D$ and $1 - D$, and $N$ further only depends on $\epsilon_{0}$ (recall the choice made in \eqref{N}). So, $\kappa$ only depends on $\epsilon_{0}$, $d - D$, $1 - D$ and the constant $C_{E}$ in \eqref{form72}. We may assume that $\Delta^{p}$ has the form
\begin{displaymath} \Delta^{p} = 2^{-\ell m} \quad \text{for some} \quad \ell \in \N. \end{displaymath}
This can be achieved by adding one more requirement for $\delta > 0$ at the start of Section \ref{preliminaries} (instead of asking that $\delta^{g^{-n}} \in 2^{-\N}$ for all $n \lesssim 1/\alpha$, we rather require that $\delta^{g^{-n}} \in 2^{-m\N}$ for the $m$ above, which only depends on $\epsilon_{0},d - D,1 - D$, and the constant $C_{E}$ in \eqref{form72}).

Then, we pick $\alpha$ so small that $C_{N}\alpha < \kappa$ in \eqref{form114}. Then, \eqref{form114} implies -- for $\Delta > 0$ small enough, and finally picking $\tau > 0$ small enough depending on $\alpha,N$ -- that the main hypothesis \eqref{inverseHyp} of Theorem \ref{shmerkin} is valid. It follows from the theorem that \eqref{form116} is valid. Then, combining \eqref{form116} and \eqref{form115}, we find that
\begin{displaymath} m|\calS| \geq \log \|(\theta \mu_{v})'\|_{L^{2},\textup{Sh}}^{-2} + \beta \log \Delta^{p} \geq (\eta - \beta) \log |D_{p}| - 3 \geq \tfrac{\eta}{2} \log |D_{p}| - 3. \end{displaymath}
If $\Delta^{p}$ is small enough, and hence $|D_{p}|$ is large enough, the inequality above implies that $m|\calS| > 0$, and hence $\calS \neq \emptyset$. (Choosing $\Delta^{p}$ small enough depending on $\eta$ is legitimate: recall that $\eta = q_{\mathrm{spec}}(d - D)/3$, and then from Section \ref{refining} that $q_{\mathrm{spec}} = 2^{-2N}$ and $p \in [2^{-N + 2},1]$, where $N$ only depends on $D - d$ and $\epsilon_{0}$. Also, recall from \eqref{nu} and \eqref{form40} that $\Delta = \delta_{n} = \delta^{\kappa}$, where $\kappa > 0$ is a constant depending only on $\alpha,Q$, and $\delta > 0$ is an "initial scale", chosen as early as in Section \ref{preliminaries}. This scale was allowed to depend on all the parameters $\alpha,d,D,\epsilon_{0},Q$. Therefore, we can arrange $\log |D_{p}| \gg 2/\eta$ by choosing $\delta > 0$ initially small enough, depending only on $\alpha,d,D,\epsilon_{0},Q$.) Recalling Theorem \ref{shmerkin}(A), it follows that there exists $s \in \mathcal{S} \subset \{0,\ldots,\ell\}$ such that
\begin{equation}\label{form117} N([\spt \mu_{h}] \cap I,2^{-(s + 1)m}) \geq R_{s}^{\mu} \geq 2^{(1 - \beta)m} \end{equation}
for some dyadic interval $I \subset \R$ of length $\ell(I) = 2^{-ms} \in [\Delta^{p},1]$. On the other hand, by \eqref{form72} applied with $R = 2^{-ms}$ and $r = 2^{-(s + 1)m}$, we find that
\begin{equation}\label{form118} N([\spt \mu_{h}] \cap I,2^{-(s + 1)m}) \lesssim C_{E} \left(\frac{2^{-ms}}{2^{-(s + 1)m}} \right)^{D} = C_{E} \cdot 2^{mD}. \end{equation} 
Finally, since $\beta < 1 - D$, we see that the inequalities \eqref{form117}-\eqref{form118} are incompatible if $m \geq m_{0}$ is sufficiently large (depending on $1 - D$ and $C_{E}$, as promised). We have reached a contradiction, and proved \eqref{form68}, namely that $|e - e_{1}| \leq \Delta^{p - \epsilon_{1}}$, and hence Proposition \ref{prop2}. As explained after \eqref{form47}, this implies the existence of the arc $J_{2} \subset J_{1}$ satisfying \eqref{form99}.

\subsection{Conclusion of the proof}\label{conclusion} We now complete the proof of Theorem \ref{mainTechnical} roughly in the way described in Section \ref{contradictionHeuristic}. We pick any initial arc $J_{1} \subset S^{1}$ of length $\ell(J_{1}) \sim \Delta^{q_{2}}$ and $\sigma(J_{1}) \gtrsim \Delta^{q_{2}}$. Then, we apply \eqref{form99} repeatedly to find a sequence of arcs $J_{1} \supset J_{2} \supset \ldots \supset J_{N - 1}$ with the properties that
\begin{itemize}
\item $\ell(J_{j}) = \Delta^{q_{j + 1}}$ for $1 \leq j \leq N - 1$, and
\item $\sigma(J_{j + 1}) \gtrapprox \Delta^{\epsilon_{1}}\sigma(J_{j})$ for $1 \leq j \leq N - 2$.
\end{itemize}
In particular $J := J_{N - 1}$ is an arc of length $\Delta$ satisfying
\begin{displaymath} \sigma(J) \gtrapprox \Delta^{N\epsilon_{1}} \sigma(J_{1}) \gtrsim \Delta^{N \cdot 2^{-N} \cdot 2^{-N + 2}} = \Delta^{(N + 3) \cdot 2^{-N}}. \end{displaymath}
On the other hand, $\sigma(J) \leq C_{\sigma}\Delta^{\epsilon_{0}}$ by \eqref{frostman}. Recalling from \eqref{N} that $10 N \cdot 2^{-N} < \epsilon_{0}$, we have reached a contradiction, assuming that $\Delta,\tau > 0$ are small enough. The proof of Theorem \ref{mainTechnical} is complete.

\bibliographystyle{plain}
\bibliography{references}

\def\cprime{$'$}
\begin{thebibliography}{10}

\bibitem{BHR}
Bal{\'a}zs {B{\'a}r{\'a}ny}, Michael {Hochman}, and Ariel {Rapaport}.
\newblock {Hausdorff dimension of planar self-affine sets and measures}.
\newblock {\em Invent. Math. (to appear)}, page arXiv:1712.07353.

\bibitem{Bo1}
J.~Bourgain.
\newblock On the {E}rd\"os-{V}olkmann and {K}atz-{T}ao ring conjectures.
\newblock {\em Geom. Funct. Anal.}, 13(2):334--365, 2003.

\bibitem{Bo2}
Jean Bourgain.
\newblock The discretized sum-product and projection theorems.
\newblock {\em J. Anal. Math.}, 112:193--236, 2010.

\bibitem{BJ}
Catherine {Bruce} and Xiong {Jin}.
\newblock {Projections of Gibbs measures on self-conformal sets}.
\newblock {\em arXiv e-prints}, page arXiv:1801.06468, January 2018.

\bibitem{FH}
K.~J. Falconer and J.~D. Howroyd.
\newblock Projection theorems for box and packing dimensions.
\newblock {\em Math. Proc. Cambridge Philos. Soc.}, 119(2):287--295, 1996.

\bibitem{FaO}
Katrin F\"{a}ssler and Tuomas Orponen.
\newblock On restricted families of projections in {$\Bbb R^3$}.
\newblock {\em Proc. Lond. Math. Soc. (3)}, 109(2):353--381, 2014.

\bibitem{FFS}
Andrew Ferguson, Jonathan~M. Fraser, and Tuomas Sahlsten.
\newblock Scaling scenery of {$(\times m,\times n)$} invariant measures.
\newblock {\em Adv. Math.}, 268:564--602, 2015.

\bibitem{Fr}
Jonathan~M. Fraser.
\newblock Distance sets, orthogonal projections and passing to weak tangents.
\newblock {\em Israel J. Math.}, 226(2):851--875, 2018.

\bibitem{FK}
Jonathan~M. {Fraser} and Antti {K{\"a}enm{\"a}ki}.
\newblock {Attainable values for the Assouad dimension of projections}.
\newblock {\em arXiv e-prints}, page arXiv:1811.00951, November 2018.

\bibitem{FO}
Jonathan~M. Fraser and Tuomas Orponen.
\newblock The {A}ssouad dimensions of projections of planar sets.
\newblock {\em Proc. Lond. Math. Soc. (3)}, 114(2):374--398, 2017.

\bibitem{Fu}
Hillel Furstenberg.
\newblock Ergodic fractal measures and dimension conservation.
\newblock {\em Ergodic Theory Dynam. Systems}, 28(2):405--422, 2008.

\bibitem{MR1800917}
Juha Heinonen.
\newblock {\em Lectures on analysis on metric spaces}.
\newblock Universitext. Springer-Verlag, New York, 2001.

\bibitem{Ho1}
Michael {Hochman}.
\newblock {Dynamics on fractals and fractal distributions}.
\newblock {\em arXiv e-prints}, page arXiv:1008.3731, August 2010.

\bibitem{Ho}
Michael Hochman.
\newblock On self-similar sets with overlaps and inverse theorems for entropy.
\newblock {\em Ann. of Math. (2)}, 180(2):773--822, 2014.

\bibitem{Ja}
Maarit J\"{a}rvenp\"{a}\"{a}.
\newblock On the upper {M}inkowski dimension, the packing dimension, and
  orthogonal projections.
\newblock {\em Ann. Acad. Sci. Fenn. Ser. A I Math. Dissertationes}, (99):34,
  1994.

\bibitem{KOR}
Antti K\"{a}enm\"{a}ki, Tuomo Ojala, and Eino Rossi.
\newblock Rigidity of quasisymmetric mappings on self-affine carpets.
\newblock {\em Int. Math. Res. Not. IMRN}, (12):3769--3799, 2018.

\bibitem{KM}
R.~Kaufman and P.~Mattila.
\newblock Hausdorff dimension and exceptional sets of linear transformations.
\newblock {\em Ann. Acad. Sci. Fenn. Ser. A I Math.}, 1(2):387--392, 1975.

\bibitem{Ka}
Robert Kaufman.
\newblock On {H}ausdorff dimension of projections.
\newblock {\em Mathematika}, 15:153--155, 1968.

\bibitem{MT}
John~M. Mackay and Jeremy~T. Tyson.
\newblock {\em Conformal dimension}, volume~54 of {\em University Lecture
  Series}.
\newblock American Mathematical Society, Providence, RI, 2010.
\newblock Theory and application.

\bibitem{Mar}
J.~M. Marstrand.
\newblock Some fundamental geometrical properties of plane sets of fractional
  dimensions.
\newblock {\em Proc. London Math. Soc. (3)}, 4:257--302, 1954.

\bibitem{Or}
Tuomas {Orponen}.
\newblock {An improved bound on the packing dimension of Furstenberg sets in
  the plane}.
\newblock {\em J. Eur. Math. Soc. (to appear)}, page arXiv:1611.09762.

\bibitem{Or1}
Tuomas Orponen.
\newblock On the packing dimension and category of exceptional sets of
  orthogonal projections.
\newblock {\em Ann. Mat. Pura Appl. (4)}, 194(3):843--880, 2015.

\bibitem{PS}
Yuval Peres and Pablo Shmerkin.
\newblock Resonance between {C}antor sets.
\newblock {\em Ergodic Theory Dynam. Systems}, 29(1):201--221, 2009.

\bibitem{Sh}
Pablo {Shmerkin}.
\newblock {On Furstenberg's intersection conjecture, self-similar measures, and
  the $L^{q}$ norms of convolutions}.
\newblock {\em Ann. of Math. (to appear)}, page arXiv:1609.07802.

\end{thebibliography}

\end{document}